%% file: Levy_Boistard_Moulines_Taqqu_Reisen_revision_new.tex
\documentclass[dvips,aos,preprint]{imsart}
\startlocaldefs
\input{def_aos}

\endlocaldefs
\begin{document}
\begin{frontmatter}
\title{Asymptotic properties of U-processes under long-range dependence}
\runtitle{U-processes under long-range dependence}
\begin{aug}
\author{\fnms{C.} \snm{L\'evy-Leduc}\thanksref{t1,m1}\ead[label=e1]{celine.levy-leduc@telecom-paristech.fr}},
\author{\fnms{H.} \snm{Boistard}\thanksref{m2}\ead[label=e2]{helene@boistard.fr}},
\author{\fnms{E.} \snm{Moulines}\thanksref{m1}\ead[label=e3]{moulines@telecom-paristech.fr}},
\author{\fnms{M. S.} \snm{Taqqu}\thanksref{m4,t2}\ead[label=e4]{murad@math.bu.edu}}
\and
\author{\fnms{V. A.} \snm{Reisen}\thanksref{m5}\ead[label=e5]{valderio@cce.ufes.br}}
\thankstext{t1}{Corresponding author}
 \thankstext{t2}{Murad S. Taqqu would like to thank Telecom ParisTech for their
  hospitality. This research was partially supported by the NSF
  Grants DMS-0706786 and DMS-1007616 at Boston University.}
\runauthor{C. L\'evy-Leduc, H. Boistard, E. Moulines, M. Taqqu and
  V. Reisen}
\affiliation{CNRS/LTCI/TelecomParisTech \thanksmark{m1}, Toulouse School of Economics \thanksmark{m2},}
\affiliation{Boston University \thanksmark{m4} and Universidade Federal do  Esp\'irito Santo \thanksmark{m5}}
\address{C\'eline L\'evy-Leduc (corresponding author)\\
CNRS/LTCI/TelecomParisTech \\
46, rue Barrault \\
75634 Paris C\'edex 13, France.\\
\printead{e1}}
\address{H\'el\`ene Boistard\\
Toulouse School of Economics \\
GREMAQ, Universit\'e Toulouse 1\\Manufacture des Tabacs\\
  b\^at. F, aile J.J. Laffont \\ 21 all\'ee de Brienne \\ 31000
  Toulouse\\\printead{e2}}
\address{Eric Moulines\\
CNRS/LTCI/TelecomParisTech \\
46, rue Barrault \\
75634 Paris C\'edex 13, France.\\
\printead{e3}}
\address{Murad S. Taqqu\\
Department of Mathematics\\ Boston University\\ 111 Cumington
  Street\\ Boston, MA 02215, USA\\\printead{e4}}
\address{Valderio Anselmo Reisen\\
Departamento de Estat\'istica\\ Universidade Federal \\do
  Esp\'irito Santo \\ Vit\'oria/ES, Brazil.\\\printead{e5}}
\end{aug}
\begin{abstract}
Let $(X_i)_{i\geq 1}$ be a stationary
mean-zero  Gaussian process with covariances $\rho(k)=\PE(X_{1}X_{k+1})$ satisfying:
$\rho(0)=1$ and $\rho(k)=k^{-D} L(k)$ where $D$ is in $(0,1)$
and $L$ is slowly varying at infinity. Consider the $U$-process
$\{U_n(r),\; r\in I\}$ defined as
$$
U_n(r)=\frac{1}{n(n-1)}\sum_{1\leq i\neq j\leq n}\1_{\{G(X_i,X_j)\leq r\}}\; ,
$$
where $I$ is an interval included in $\rset$ and $G$ is a
symmetric function.
In this paper, we provide central and non-central limit theorems for
$U_n$. 
They are used to
derive, in the long-range dependence setting, new 
properties
of many well-known estimators such as the Hodges-Lehmann estimator, which is a well-known
robust location estimator, the Wilcoxon-signed rank statistic,
the sample correlation integral and an associated robust scale
estimator. These robust estimators are shown to
have the same asymptotic distribution as the classical location and
scale estimators.
The limiting distributions are expressed through multiple Wiener-It\^o
integrals.

\end{abstract}
\end{frontmatter}
\section{Introduction}
Since the seminal work by \cite{hoeffding:1948}, $U$-statistics
have been widely studied to investigate the asymptotic properties
of many statistics such as the sample variance, the Gini's mean
difference and the Wilcoxon one-sample statistic, see
\cite{serfling:1980} for other examples.
One of the most powerful tools used to derive the asymptotic
behavior of $U$-statistics is the Hoeffding's decomposition [\cite{hoeffding:1948}].
In the i.i.d and weak dependent
frameworks, it provides a decomposition of a $U$-statistic into
several terms having different orders of magnitudes, and in general
the one with the leading order determines
the asymptotic behavior of the $U$-statistic, see \cite{serfling:1980},
\cite{borovkova:burton:dehling:2001} and the references therein
for further details. A recent review of the properties of
$U$-statistics in various frameworks
is presented in \cite{hsing:wu:2004}.
In the case of processes having a long-range dependent structure,
decomposition ideas are also crucial. 
However, in the case of Gaussian long-memory processes, the classical
Hoeffding's decomposition may not provide the complete
asymptotic behavior of
$U$-statistics because all terms of
this decomposition may contribute
to the limit, see for example
\cite{dehling:taqqu:1991}. In this case, the asymptotic study
of $U$-statistics can be achieved
by using an expansion in Hermite polynomials, see
\cite{Dehling:Taqqu:1989,dehling:taqqu:1991}.
For a large class of processes including linear and nonlinear processes,
a new decomposition is discussed in \cite{hsing:wu:2004}.
These authors use martingale-based
techniques to establish the asymptotic
properties of $U$-statistics.

A very natural extension of $U$-statistics (which are random
variables) is the notion of $U$-processes
which encompasses a wide class of estimators. For example,
\cite{borovkova:burton:dehling:2001} study the
Grassberger-Proccacia estimator which can be used to estimate the
correlation dimension. In Section 5 of their work, the authors
investigate the asymptotic properties of $U$-processes when the underlying
observations are functionals of an absolutely regular process,
that is, short-memory processes.
As far as we know, the asymptotic properties of $U$-processes
in the case of long-range dependence setting have not been established yet,
and this is the heart of the research discussed in this paper.
More precisely, our contribution consists first in
extending the results of \cite{borovkova:burton:dehling:2001} in order to
address the long-range dependence case, second in extending the
results obtained in \cite{Dehling:Taqqu:1989} to functions of two variables
and third in extending the results of \cite{hsing:wu:2004}
to $U$-processes. The authors of the latter paper establish the asymptotic 
properties of $U$-statistics involving causal but non necessarily Gaussian long-range dependent processes 
whereas, in our paper, we establish the asymptotic properties of $U$-processes
involving Gaussian long-range dependent processes. The authors in \cite{hsing:wu:2004} use a martingale decomposition
and we use a Hoeffding decomposition or a decomposition in Hermite polynomials.
In the proof section, we also present an extension of some
results of \cite{soulier:2001}.

Consider the $U$-process defined by
\begin{equation}\label{def:U_n}
U_n(r)=\frac{1}{n(n-1)}\sum_{1\leq i\neq j\leq n}\1_{\{G(X_i,X_j)\leq
  r\}}\; ,\quad r\in I
\end{equation}
where $I$ is an interval included in $\rset$, $G$ is a symmetric function \textit{i.e.}
$G(x,y)=G(y,x)$ for all $x,y$ in $\rset$, and the process $(X_i)_{i\geq 1}$
satisfies the following assumption:
\begin{hypA}
\label{assum:long-range} $(X_i)_{i\geq 1}$ is a stationary
mean-zero  Gaussian process with covariances $\rho(k)=\PE(X_{1}X_{k+1})$ satisfying:
$$
\rho(0)=1 \textrm{ and } \rho(k)=k^{-D} L(k),\ 0<D<1\; ,
$$
where $L$ is slowly varying at infinity and is positive for large $k$.
\end{hypA}
Note that, for a fixed $r$, $U_n(r)$ is a $U$-statistic based on the kernel $h(\cdot,\cdot, r)$ where
\begin{eqnarray}\label{def:h}
h(x,y,r)=\1_{\{G(x,y)\leq r\}}\; , \forall x,y\in\rset \textrm{
    and } r\in I\; .
\end{eqnarray}
We show in this paper that the asymptotic properties of the
$U$-process $U_n(\cdot)$ depends on the
value of $D$ and on the Hermite rank $m$ of the class of functions
$\{h(\cdot,\cdot,r)-U(r), r\in I\}$, defined in Section
\ref{sec:main}. We obtain the rate of convergence of
$U_n(\cdot)$ and also provide the
limiting process when $D>1/2$, $m=2$  and when
$D<1/m$, $m=1,2$. The convergence
rate in the former case is of order $\sqrt{n}$ whereas it is
of order $n^{mD/2}/L(n)^{m/2}$ in the latter.
These results are stated in Theorems \ref{theo:U_n_D>1/2} and
\ref{theo:D<1/2}, respectively.
They are applied to derive the asymptotic properties
of well-known robust location and scale estimators such as the
Hodges-Lehmann estimator [\cite{hodges:lehmann:1963}]
and the Shamos scale estimator proposed by \cite{shamos:1976}
and analyzed by \cite{bickel:lehmann:1979}. These properties are illustrated in
\cite{levy;boistard:moulines:taqqu:reisen:c} using numerical experiments.
Theorems \ref{theo:U_n_D>1/2} and
\ref{theo:D<1/2} allow us to establish novel asymptotic properties on these estimators
in the long-range dependence context. The most striking result is that these robust estimators
have the same asymptotic distribution as the classical estimators, see
Propositions \ref{prop:hodges-lehmann} and \ref{prop:shamos_scale}
in Section \ref{sec:appli}.

Theorems \ref{theo:U_n_D>1/2} and
\ref{theo:D<1/2} have also been used to derive the asymptotic
distribution of a robust scale estimator proposed by \cite{Rousseeuw:Croux:1993}  and a robust autocovariance
estimator introduced in \cite{Genton:Ma:2000}. The robustness and efficiency properties of
these estimators have also been investigated through numerical experiments
and real data analysis. For further details on these theoretical and
numerical studies, we refer the reader to \cite{levy;boistard:moulines:taqqu:reisen:b}.

The paper is organized as follows.  In Section \ref{sec:main},
the main theorems \ref{theo:U_n_D>1/2} and \ref{theo:D<1/2} are stated.
In Section \ref{s:U}, we derive the asymptotic properties of
some quantile estimators.
Section \ref{sec:appli} presents new asymptotic results
in the context of long-range dependence.
In this section, central and non-central limit theorems
are provided for several statistics as an illustration of
the theory presented in Sections \ref{sec:main} and \ref{s:U}.
These statistics are the Hodges-Lehmann
estimator [\cite{hodges:lehmann:1963}],
the Wilcoxon-signed rank statistic [\cite{wilcoxon:1945}],
the sample correlation integral [\cite{grassberger:procaccia:1983}]
and an associated scale estimator proposed by \cite{shamos:1976}
and\\ \cite{bickel:lehmann:1979}.
Section \ref{sec:proofs} develops the proofs of the
results stated in Section \ref{sec:main}. A supplemental article 
\cite{supplement} contains the proofs of some of the lemmas.
It contains also Section \ref{sec:exp} which concerns numerical
experiments.

\section{Main results}\label{sec:main}
We start by introducing the terms involved in the
Hoeffding's decomposition [\cite{hoeffding:1948}].
Recall the definition of $U_n(\cdot)$ in (\ref{def:U_n}) and let
$U(\cdot)$ be defined as
\begin{eqnarray}\label{def:U_hoeff}
U(r)=\int_{\rset^2} h(x,y,r)\varphi(x)\varphi(y)
\rmd x \rmd y\; , \textrm{ for all } r \textrm{ in } I\; ,
\end{eqnarray}
where $\varphi$ denotes the p.d.f of a standard Gaussian random
variable and $h$ is given by (\ref{def:h}). For all $x$ in $\rset$,
and $r$ in $I$, let us define
\begin{eqnarray}\label{def:h_1}
  h_1(x,r)=\int_{\mathbb{R}} h(x,y,r) \varphi(y)\rmd y\; .
\end{eqnarray}
The Hoeffding decomposition amounts to expressing, for all $r$ in $I$,
the difference
\begin{equation}\label{e:Uh}
U_n(r)-U(r)=\frac{1}{n(n-1)}\sum_{1\leq i\neq j\leq n}
[h(X_i,X_j,r)-U(r)]\; ,
\end{equation}
as
\begin{equation}\label{eq:hoeff_decomp}
U_n(r)-U(r)=W_n(r)+R_n(r)\; ,
\end{equation}
where
\begin{equation}\label{eq:W_n_hoeff}
W_n(r)=\frac{2}{n}\sum_{i=1}^n\left\{h_1(X_i,r)-U(r)\right\}\; ,
\end{equation}
and
\begin{equation}\label{eq:R_n_hoeff}
R_n(r)=\frac{1}{n(n-1)}\sum_{1\leq i\neq j\leq
  n}\left\{h(X_i,X_j,r)-h_1(X_i,r)-h_1(X_j,r)+U(r)\right\}\; .
\end{equation}
We now define  the
Hermite rank of the class of functions $\{h(\cdot,\cdot,r)-U(r), r\in I\}$
which plays a crucial role in understanding the asymptotic behavior of the
$U$-process $U_n(\cdot)$. We shall expand the function $(x,y)\mapsto h(x,y,r)
$ in a Hermite polynomials basis of $L^2_{\varphi}(\rset^2)$, that is, the $L^2$
space on $\rset^2$ equipped with product standard Gaussian measures. We use Hermite polynomials
with leading coefficients equal to one which are:
$H_0(x)=1$, $H_1(x)=x$, $H_2(x)=x^2-1$, $H_3(x)=x^3-3x,\dots$ . We get
\begin{equation}\label{eq:hermite_decomp}
h(x,y,r)=\sum_{p,q\geq 0}\frac{\alpha_{p,q}(r)}{p!q!} H_p(x) H_q(y)\;
, \textrm{ in } L^2_{\varphi}(\rset^2)\; ,
\end{equation}
where
\begin{equation}\label{e:alpha}
\alpha_{p,q}(r)=\PE[h(X,Y,r)H_p(X) H_q(Y)]\; ,
\end{equation}
and where $(X,Y)$ is a \textit{standard Gaussian vector} that is $X$
and $Y$ are independent standard Gaussian random variables.
Thus,
\begin{equation}\label{e:hsquare}
\PE [h^2(X,Y,r)]=\sum_{p,q\geq 0}\frac{\alpha^2_{p,q}(r)}{p!q!}\; .
\end{equation}
Note that $\alpha_{0,0}(r)$ is equal to $U(r)$ for all $r$, where $U(r)$ is defined in
(\ref{def:U_hoeff}). The Hermite rank of $h(\cdot,\cdot , r)$ is
the smallest positive integer $m(r)$ such that there exist $p$ and $q$
satisfying $p+q=m(r)$ and $\alpha_{p,q}(r)\neq 0$.
Thus, \eqref{eq:hermite_decomp} can be rewritten
as
\begin{equation}\label{eq:hermite_rank}
h(x,y,r)-U(r)=\sum_{\stackrel{p,q\geq 0}{p+q\geq m(r)}}\frac{\alpha_{p,q}(r)}{p!q!}
H_p(x) H_q(y)\; , \textrm{ in } L^2_{\varphi}(\rset^2)\; .
\end{equation}
%
%
The Hermite rank $m$ of the class of functions $\{h(\cdot,\cdot,r)-U(r)\; , r\in I\}$ is the
smallest index $m=p+q\geq 1$ such that $\alpha_{p,q}(r)\neq 0$ for at least one $r$
in $I$, that is, $m=\inf_{r\in I} m(r)$.
By integrating with respect to $y$ in
(\ref{eq:hermite_decomp}), we obtain the expansion in Hermite polynomials of $h_1$
as a function of $x$:
\begin{equation}\label{eq:hermite_decomp_h_1}
h_1(x,r)-U(r)=\sum_{p\geq 1}\frac{\alpha_{p,0}(r)}{p!} H_p(x)\;
, \textrm{ in } L^2_{\varphi}(\rset)\;,
\end{equation}
where $L^2_{\varphi}(\rset)$ denotes the $L^2$ space
  on $\rset$ equipped with the standard Gaussian measure.
Let $\tau(r)$ be the smallest integer greater than or equal to
$1$ such that $\alpha_{\tau,0}(r)\neq 0$, that is, the Hermite
rank of the function $h_1(\cdot,r)-U(r)$.
The Hermite rank of the class of functions $\{h_1(\cdot,r)-U(r)\; , r\in
I\}$ is the smallest index $\tau\geq 1$ such that
$\alpha_{\tau,0}(r)\neq 0$
for at least one $r$. Since $\tau(r)\geq m(r)$, for all $r$ in $I$,
one has
\begin{equation}\label{eq:tau}
\tau\geq m\; .
\end{equation}
In the sequel, we shall assume that $m$ is equal to 1 or 2. As shown
in Section \ref{sec:appli}, this covers most of the situations
of practical interest.
Theorem \ref{theo:U_n_D>1/2}, given below, establishes the central-limit theorem for the $U$-process
$\{\sqrt{n}(U_n(r)-U(r)), r\in I\}$ when
$$
D>1/m \textrm{ and } m=2\; .
$$
\begin{theo}\label{theo:U_n_D>1/2}
Let $I$ be a compact interval of $\rset$.
Suppose that the Hermite rank of the class of functions $\{h(\cdot,\cdot,r)-U(r)\; ,
r\in I\}$ as defined in (\ref{eq:hermite_rank}) is $m=2$ and
that Assumption (A\ref{assum:long-range}) is satisfied
with $D>1/2$. Assume that $h$ and $h_1$, defined in (\ref{def:h}) and (\ref{def:h_1}),
satisfy the three following conditions:
\begin{enumerate}[(i)]
\item There exists a positive constant $C$
such that for all  $s$, $t$ in $I$, $u$, $v$ in $\rset$,
\begin{equation}\label{cond:h_suppl}
\PE[|h(X+u,Y+v,s)-h(X+u,Y+v,t)|]\leq\; C|t-s|\; ,
\end{equation}
where $(X,Y)$ is a standard Gaussian vector.
\item There exists a positive constant $C$
such that for all $k\geq 1$,
\begin{equation}\label{cond:contr_G}
\PE[|h(X_1+u,X_{1+k}+v,t)-h(X_1,X_{1+k},t)|]\leq\; C(|u|+|v|)\; ,
\end{equation}
\begin{equation}\label{eq:G_lip}
\PE[|h(X_1,X_{1+k},s)-h(X_1,X_{1+k},t)|]\leq C |t-s|\; .
\end{equation}
\item There exists a positive constant $C$
such that
for all  $t$, $s$ in $I$, and $x$, $u$, $v$ in $\rset$,
\begin{equation}\label{cond:h_1}
|h_1(x+u,t)-h_1(x+v,t)|\leq C(|u|+|v|)\; ,
\end{equation}
and
\begin{equation}\label{assum:h_1_Lipschitz}
|h_1(x,s)-h_1(x,t)|\leq C|t-s|\; .
\end{equation}
\end{enumerate}
Then the $U$-process
$$
\{\sqrt{n}(U_n(r)-U(r)), r\in I\}
$$
defined in
(\ref{def:U_n}) and (\ref{def:U_hoeff})
 converges weakly in the space of cadlag functions  $\mathcal{D}(I)$ equipped with the
topology of uniform convergence to the zero mean Gaussian process $\{W(r),r\in I\}$ with
covariance structure given by
\begin{multline}\label{eq:cov_struct}
\PE[W(s)W(t)]=4\;\Cov(h_1(X_1,s),h_1(X_1,t))\\
+4\sum_{\ell\geq1}\{\Cov(h_1(X_1,s),h_1(X_{\ell+1},t))+\Cov(h_1(X_1,t),h_1(X_{\ell+1},s))\}\; .
\end{multline}
\end{theo}
\begin{proof}[Proof of Theorem \ref{theo:U_n_D>1/2}]
The proof of the theorem follows from the decomposition
(\ref{eq:hoeff_decomp}) and Lemmas \ref{lem:O_P(1)} and
\ref{lem:reste1}, given in Section \ref{subsec:theo1}. Lemma
\ref{lem:O_P(1)} states that $\{\sqrt{n} W_n(r), r\in I\}$ converges weakly in the space of cadlag
functions $\mathcal{D}(I)$ equipped with the topology of uniform
convergence. Lemma \ref{lem:reste1} states that $\sup_{r\in I} \sqrt{n}
|R_n(r)|=o_P(1)$. Its proof uses Lemmas \ref{eq:maj_Rn}, 
\ref{lem:controle_hermite} and \ref{lem:ext:soulier}.
\end{proof}

\begin{remark}
The examples of Section 4 satisfy the conditions (15) to (19), \textit{e.g.} through
the choice $G(x,y)=(x+y)/2$. More generally, suppose either:\\
(i) $G$ is linear.\\
(ii) The function $G$ can be written as $G(x,y)=g(L(x,y))$
where $L(x,y)=\alpha x + \beta y$ is some linear function of $(x,y)$, 
$\alpha$ and $\beta$ in $\mathbb{R}$ are such that
$|\alpha|=|\beta|$ and 
$g$ is an even function satisfying for some $\lambda_g>0$:
$\forall x,\; t\leq g(x)\leq s \implies \lambda_g t \leq |x|\leq \lambda_g s.$\\
(iii) $G\geq 0$ and satisfies the triangle inequality: 
$G(x+x',y+y')\leq G(x,y)+G(x',y')$
and there exists some constant $C$ such that for all $(x,y)$, 
$G(x,y)\leq C(|x|+|y|).$\\
Then Condition (i) implies Conditions (15) to (19),
Condition (ii) implies Conditions (15), (17) and (19) and 
Condition (iii) implies Conditions (16) and (18). 
The proofs are based on techniques similar to the verification of (27) in Section 4.1 
where $G(x,y)=(x+y)/2$. 
\end{remark}

\begin{remark}
The set $I$ in the previous theorem may be equal to
$[-\infty,+\infty]$ which involves the two-point compactification of
the real line. Since $[-\infty,+\infty]$ is compact, all functions
in $\mathcal{D}([-\infty,+\infty])$ are bounded. In fact, that space is
isomorphic to $\mathcal{D}[0,1]$.
\end{remark}
When $D<1/m$, $W_n$ and $R_n$ are not the leading
term and the remainder term, respectively.
Note that, on one hand, for a fixed $r$,
Corollary 2 of \cite{Dehling:Taqqu:1989} gives $R_n(r)=O_P(n^{-D}L(n))$
for any $D$ in $(0,1)$. On the other hand, if $D<1/\tau$,
where $\tau$ is defined in (\ref{eq:tau}), Theorem 6 of
\cite{Arcones:1994} implies that $W_n(r)=O_P( n^{-\tau D/2}L(n)^{\tau/2})$
and if $D$ is in $(1/\tau,1/m)$, $W_n(r)=O_P(n^{-1/2})$ by Theorem 4 of \cite{Arcones:1994}.
Thus, if for instance, $\tau=m=2$, $W_n(r)$ and $R_n(r)$ may be of the
same order $O_P(n^{-D}L(n))$.
Hence, to study the case $D<1/m$, we shall introduce a different
decomposition of $U_n(\cdot)$
based on the expansion of $h$ in the basis of Hermite polynomials given by (\ref{eq:hermite_decomp}).
Thus, $U_n(r)$ defined in (\ref{def:U_n})
can be rewritten as follows
\begin{equation}\label{eq:herm_decomp_Un}
n(n-1) \{U_n(r)-U(r)\}=\widetilde{W}_n(r)+\widetilde{R}_n(r)\; ,
\end{equation}
where
\begin{equation}\label{def:W_n}
\widetilde{W}_n(r)=\sum_{1\leq i\neq j\leq n} \sum_{\stackrel{p,q\geq 0}{p+q\leq m}}\frac{\alpha_{p,q}(r)}{p!q!}
H_p(X_i) H_q(X_j)\; .
\end{equation}
Introduce also the Beta function
\begin{equation}\label{eq:betaf}
\B(\alpha,\beta)=\int_0^{\infty} y^{\alpha-1}(1+y)^{-\alpha-\beta}\rmd
y = \frac{\Gamma(\alpha)\Gamma(\beta)}{\Gamma(\alpha+\beta)}\; ,
\quad\alpha >0,\; \beta>0\; .
\end{equation}
The limit processes which appear in the next theorem are  the standard fractional Brownian motion
(fBm) $(Z_{1,D}(t))_{0\leq t\leq 1}$ and the Rosenblatt process
$(Z_{2,D}(t))_{0\leq t\leq 1}$.
They are defined through multiple Wiener-It\^o  integrals and given by
\begin{equation}\label{eq:fBm}
Z_{1,D}(t)=\int_{\rset}\left[\int_0^t (u-x)_+^{-(D+1)/2} \rmd u\right] \rmd B(x),\quad 0<D<1\; ,
\end{equation}
and
\begin{equation}\label{eq:rosenblatt}
Z_{2,D}(t)=\int'_{\rset^2}\left[\int_0^t (u-x)_+^{-(D+1)/2}
(u-y)_+^{-(D+1)/2} \rmd u\right] \rmd B(x) \rmd B(y),\quad 0<D<1/2\; ,
\end{equation}
where $B$ is the standard Brownian motion, see \cite{fox:taqqu:1987}.
The symbol $\int'$ means that the domain of integration excludes the
diagonal. Note that $Z_{1,D}$ and $Z_{2,D}$ are dependent but uncorrelated.
The following theorem treats the case $D<1/m$ where $m=1$ or 2.
\begin{theo}\label{theo:D<1/2}
Let $I$ be a compact interval of $\rset$.
Suppose that Assumption (A\ref{assum:long-range}) holds with $D<1/m$, where $m=1$
or 2 is
the Hermite rank of the class of functions $\{h(\cdot,\cdot,r)-U(r)\; ,
r\in I\}$ as defined in (\ref{eq:hermite_rank}).
Assume the following:
\begin{enumerate}[(i)]
\item There exists a positive constant C such that,
for all $k\geq 1$ and for all $s,t$ in $I$,
\begin{equation}\label{eq:lip_h}
\PE[|h(X_1,X_{1+k},s)-h(X_1,X_{1+k},t)|]\leq C |t-s|\; .
\end{equation}
\item $U$ is a Lipschitz function.
\item The function $\widetilde{\Lambda}$ defined, for all $s$ in $I$, by
\begin{equation}\label{eq:lambda_tilde}
\widetilde{\Lambda}(s)=\PE[h(X,Y,s)(|X|+|XY|+|X^2-1|)]\; ,
\end{equation}
where $X$ and $Y$ are independent standard Gaussian random variables,
is also a Lipschitz function.
\end{enumerate}
Then,
$$
\left\{n^{mD/2}L(n)^{-m/2}\left(U_n(r)-U(r)\right); r\in I\right\}
$$
converges weakly in the space of cadlag functions $\mathcal{D}(I)$, equipped with the
topology of uniform convergence, to
$$
\{2\alpha_{1,0}(r)k(D)^{-1/2}Z_{1,D}(1); r\in I\}\; , \quad\textrm{if } m=1\;,
$$
and to
$$
\{k(D)^{-1}\left[\alpha_{1,1}(r)Z_{1,D}(1)^2+\alpha_{2,0}(r)Z_{2,D}(1)\right];
r\in I\}\; , \quad\textrm{if } m=2\;,
$$
where the fractional Brownian motion $Z_{1,D}(\cdot)$ and the
Rosenblatt process $Z_{2,D}(\cdot)$ are defined in (\ref{eq:fBm}) and
(\ref{eq:rosenblatt}) respectively and where
\begin{equation}\label{e:kD}
k(D)=\emph{\B}((1-D)/2,D)\; ,
\end{equation}
where $\emph{\B}$ is the Beta
function defined in (\ref{eq:betaf}).
\end{theo}
The proof of Theorem \ref{theo:D<1/2} is given in Section \ref{subsec:theo2:proof}.
\begin{remark}
The processes $Z_{1,D}$ and $Z_{2,D}$ are self-similar with mean
0. They are, however, not normalized. One has
$$
\PE[Z_{1,D}(t) Z_{1,D}(s)]=\PE[Z^2_{1,D}(1)] \frac12
\left\{t^{2H_1}+s^{2H_1}-|t-s|^{2H_1}\right\}\; ,
$$
$$
\PE[Z_{2,D}(t) Z_{2,D}(s)]=\PE[Z^2_{2,D}(1)] \frac12
\left\{t^{2H_2}+s^{2H_2}-|t-s|^{2H_2}\right\}\; ,
$$
where $H_1=1-D/2\in (0,1/2)$, $H_2=1-D\in (0,1/2)$ and
\begin{equation}\label{e:VarFBMa}
\PE[Z^2_{1,D}(1)]=\frac{2k(D)}{(-D+1)(-D+2)}\; ,
\end{equation}
\begin{equation}\label{e:VarRos}
\PE[Z^2_{2,D}(1)]=\frac{4k(D)^2}{(-2D+1)(-2D+2)}\;,
\end{equation}
with $k(D)$ defined by (\ref{e:kD}). See Remark \ref{r:moments}
below for justification. The non-Gaussian random variables
$Z^2_{1,D}(1)$ and $Z_{2,D}(1)$ are dependent. Their joint cumulants
are given in (\ref{eq:cumulants}) in the supplemental article \cite{supplement}.
\end{remark}
\begin{remark}
The results of Theorem \ref{theo:D<1/2} can be extended to
the two-parameter process
$\{U_{[nt]}(r)-U(r); r \in I, 0\leq t\leq 1\}$. One can show that
$$
\left\{\frac{n^{mD/2}}{L(n)^{m/2}}\left(U_{[nt]}(r)-U(r)\right);\; r\in I, 0\leq t\leq 1\right\}
$$
converges weakly in $\mathcal{D}(I\times [0,1])$, equipped with the
topology of uniform convergence, to
$$
\{2\alpha_{1,0}(r)k(D)^{-1/2}Z_{1,D}(t);\; r \in I, 0\leq t\leq 1\}\; ,\quad
\textrm{if } m=1\;,
$$
and to
$$
\{k(D)^{-1}\left[\alpha_{1,1}(r)Z_{1,D}(t)^2+\alpha_{2,0}(r)Z_{2,D}(t)\right];\;
r \in I, 0\leq t\leq 1\}\; , \quad\textrm{if } m=2\;.
$$
\end{remark}
\section{Asymptotic behavior of empirical quantiles}
\label{s:U}
We shall apply Theorems \ref{theo:U_n_D>1/2} and \ref{theo:D<1/2}
in the preceding section to empirical quantiles.
Recall that if $V:I\longrightarrow [0,1]$ is a non-decreasing cadlag
function, where $I$ is an interval of $\rset$, then its generalized
inverse $V^{-1}$ is defined by
$V^{-1}(p)=\inf\{r\in I,\; V(r)\geq p\}$.
This applies to $U_n(r)$ and $U(r)$
since these are non-decreasing functions of $r$.
We derive in the following corollaries the asymptotic
behavior of the empirical quantile $U_n^{-1}(\cdot)$ using Theorems
\ref{theo:U_n_D>1/2} and \ref{theo:D<1/2}.
\begin{corollary}\label{coro:quantile_1}
Let $p$ be a fixed real number in $(0,1)$. Assume that the conditions of Theorem
\ref{theo:U_n_D>1/2} are satisfied. Suppose also that there exists
some $r$ in $I$ such that $U(r)=p$, that $U$ is differentiable at
$r$ and that $U'(r)$ is non null. Then, as $n$ tends to infinity,
$$
\sqrt{n}(U_n^{-1}(p)-U^{-1}(p))\stackrel{d}{\longrightarrow}
-W(U^{-1}(p))/U'(U^{-1}(p))\; ,
$$
where $W$ is a Gaussian process having a covariance structure given by
(\ref{eq:cov_struct}).
\end{corollary}
\begin{proof}[Proof of Corollary \ref{coro:quantile_1}]
By Lemma 21.3  in \cite{vdv:2000}, the functional $T: V\mapsto V^{-1}(p)$
is Hadamard differentiable at $V$ tangentially to the set of functions $h$ in $\mathcal{D}([0,1])$
with derivative $T'_{V}(h)=-h(V^{-1}(p))/V'(V^{-1}(p))$.
Applying the functional Delta method (Theorem 20.8 in \cite{vdv:2000})
thus yields
\begin{multline*}
\sqrt{n}(U_n^{-1}(p)-U^{-1}(p))=T'_{U}\{\sqrt{n}(U_n-U)\}+o_P(1)\\
=-\sqrt{n}\frac{(U_n-U)(U^{-1}(p))}{U'(U^{-1}(p))}+o_P(1)\; .
\end{multline*}
The corollary then follows from Theorem \ref{theo:U_n_D>1/2}.
\end{proof}
\begin{corollary}\label{coro:quantile_2}
Let $p$ be a fixed real number in $(0,1)$. Assume that the conditions of Theorem
\ref{theo:D<1/2} are satisfied.
Suppose also that there exists
some $r$ in $I$ such that $U(r)=p$, that $U$ is differentiable at
$r$ and that $U'(r)$ is non null.
Then, as $n$ tends to infinity,
$$
\frac{n^{mD/2}}{L(n)^{m/2}}(U_n^{-1}(p)-U^{-1}(p))
$$
converges in distribution to
$$
-2k(D)^{-1/2}\frac{\alpha_{1,0}(U^{-1}(p))}{U'(U^{-1}(p))}Z_{1,D}(1)\; ,\textrm{ if }
m=1\; ,
$$
and to
$$
-k(D)^{-1}\left\{\alpha_{1,1}(U^{-1}(p))Z_{1,D}(1)^2+\alpha_{2,0}(U^{-1}(p))Z_{2,D}(1)\right\}/U'(U^{-1}(p))\;,
\textrm{ if } m=2\; ,
$$
where $Z_{1,D}(\cdot)$ and $Z_{2,D}(\cdot)$ are defined in (\ref{eq:fBm}) and
(\ref{eq:rosenblatt}) respectively, $k(D)$ in (\ref{e:kD}) and $\alpha_{p,q}(\cdot)$ is defined
in (\ref{e:alpha}).
\end{corollary}
The proof of Corollary \ref{coro:quantile_2} is based on similar
arguments as the proof of Corollary \ref{coro:quantile_1} and is
thus omitted.
\section{Applications}\label{sec:appli}
We shall use the results established
in Sections \ref{sec:main} and \ref{s:U} to study the asymptotic properties
of several estimators based on $U$-processes in the long-range
dependence setting.
\subsection{Hodges-Lehmann estimator}
Consider the problem of estimating the location parameter
of a long-range dependent Gaussian process. Assume that
$(Y_i)_{i\geq 1}$ satisfy $Y_i=\theta+X_i$ where
$(X_i)_{i\geq 1}$ satisfy Assumption (A\ref{assum:long-range}).
To estimate the location parameter $\theta$,
\cite{hodges:lehmann:1963} suggest using the median of the average
of all pairs of observations. The statistic they propose is
\begin{multline*}
\hat{\theta}_{HL}=\textrm{median}\left\{\frac{Y_i+Y_j}{2}; 1\leq
  i<j\leq n\right\}\\
=\theta+\textrm{median}\left\{\frac{X_i+X_j}{2}; 1\leq
  i<j\leq n\right\}\; .
\end{multline*}
Define the $U$-process $U_n(r)$, $r\in\rset$ by (\ref{def:U_n}), where
$G(x,y)=(x+y)/2$.
The Hodges-Lehmann estimator may be then expressed as
$$
\hat{\theta}_{HL}=\theta+U^{-1}_n(1/2)\; .
$$
If $A$ and $B$ are independent standard Gaussian variables,
\begin{equation}\label{e:alpha_HL}
\alpha_{1,0}(r)=\alpha_{0,1}(r)=\PE[A\1_{\{A+B\leq 2r\}}]
=-\int_{\mathbb{R}}\varphi(2r-y)\varphi(y)\rmd y=-\varphi(r\sqrt{2})/\sqrt{2}\; ,
\end{equation}
using $x\varphi(x)=-\dot{\varphi}(x)$, where $\dot{\varphi}$ denotes
the first derivative of $\varphi$. The quantities in
(\ref{e:alpha_HL}) are different from 0 for all $r$ in $\rset$ since $\varphi$
is the p.d.f of a standard Gaussian random variable. Thus, the Hermite
rank $m$ of the class of functions $\{\1_{G(\cdot,\cdot)\leq r}-\alpha_{0,0}(r); r\in\rset\}$
is equal to 1.
In order to derive the asymptotic properties of $\hat{\theta}_{HL}$,
we now check the conditions of Theorem \ref{theo:D<1/2}.
Let us check Condition  (\ref{eq:lip_h}).
Note that for all $k\geq 1$, $X_1+X_{1+k}\sim\mathcal{N}(0,2(1+\rho(k)))$, thus if
$t\leq s$,
\begin{multline*}
\PE[h(X_1,X_{1+k},s)-h(X_1,X_{1+k},t)]=\Phi\left(\frac{\sqrt{2}s}{\sqrt{1+\rho(k)}}\right)
-\Phi\left(\frac{\sqrt{2}t}{\sqrt{1+\rho(k)}}\right)\\
\leq \frac{1}{\sqrt{\pi}}\frac{|t-s|}{\sqrt{1+\rho_{\star}}}\; ,
\end{multline*}
where $\Phi$ is the c.d.f of a standard Gaussian random variable
and $\rho_{\star}=\inf_{k}\rho(k)>-1$. Hence (\ref{eq:lip_h}) holds.
Similarly,
$|U(s)-U(t)|\leq|\Phi(\sqrt{2}s)-\Phi(\sqrt{2}t)|\leq\pi^{-1/2}|t-s|$
and hence $U$ is a Lipschitz function.
Let us now check Condition (\ref{eq:lambda_tilde}).
Note that, if $s\leq t$
\begin{multline*}
\int\int \1_{\{s<x+y\leq t\}}(|x|+|xy|+|x^2-1|)\varphi(x)\varphi(y) \rmd x
\rmd y =\\
\int\left(\int_{s-x}^{t-x}\varphi(y) \rmd y\right)|x|\varphi(x) \rmd x
+\int\left(\int_{s-x}^{t-x}|y|\varphi(y) \rmd y\right)|x|\varphi(x) \rmd x\\
+\int\left(\int_{s-x}^{t-x}\varphi(y) \rmd y\right)|x^2-1|\varphi(x) \rmd x\; .
\end{multline*}
Using that $\varphi(\cdot)$ and $|.|\varphi(\cdot)$ are bounded and that
the moments of Gaussian random variables are all finite, we get (\ref{eq:lambda_tilde}).
The assumptions of Theorem \ref{theo:D<1/2} are thus satisfied with
$m=1$ and hence we get that
$$
\left\{n^{D/2}L(n)^{-1/2}\left(U_n(r)-U(r)\right); -\infty\leq r\leq +\infty\right\}
$$
converges weakly in $D([-\infty,+\infty])$, equipped with the
sup-norm, to
$$
\{-\sqrt{2} k(D)^{-1/2}\varphi(r\sqrt{2}) Z_{1,D}(1); -\infty\leq r\leq +\infty\}\; .
$$
Here, $U(r)=\int\Phi(2r-x)\varphi(x) \rmd x$,
$U'(r)=2\int\varphi(2r-x)\varphi(x) \rmd x$,
$U(0)=1/2\int(\Phi(x)+\Phi(-x))\varphi(x)\rmd x=1/2$, $U^{-1}(1/2)=0$ and
$U'(U^{-1}(1/2))=U'(0)=1/\sqrt{\pi}$. Since, by (\ref{e:alpha_HL}),
$\alpha_{1,0}(U^{-1}(1/2))=\alpha_{1,0}(0)=-(2\sqrt{\pi})^{-1}$,
Corollary \ref{coro:quantile_2} implies that
\begin{equation}\label{e:limitHL}
n^{D/2} L(n)^{-1/2}(\hat{\theta}_{HL}-\theta)\stackrel{d}{\longrightarrow}
k(D)^{-1/2}Z_{1,D}(1)\; ,
\end{equation}
where using (\ref{e:VarFBMa}), $k(D)^{-1/2}Z_{1,D}(1)$ is a zero-mean Gaussian random
variable with variance $2(-D+1)^{-1}(-D+2)^{-1}$.
Let's now compare the asymptotic behavior of the Hodges-Lehmann
estimator with that of the sample mean.
Lemma 5.1
in \cite{taqqu:1975} shows that the sample mean $\bar{Y}_n=n^{-1}\sum_{i=1}^n
Y_i$ satisfies the following central limit theorem
$$
n^{D/2} L(n)^{-1/2}(\bar{Y}_n-\theta)
\stackrel{d}{\longrightarrow}k(D)^{-1/2}Z_{1,D}(1)\; .
$$
We have thus proved
\begin{prop}\label{prop:hodges-lehmann}
In the long-memory framework with $0<D<1$, the asymptotic behavior
of the Hodges-Lehmann estimator is Gaussian and given by
(\ref{e:limitHL}). It converges to $\theta$ at the same rate as the
sample mean with the same limiting distribution. There is no loss
of efficiency.
\end{prop}

A similar result was proved in \cite{beran:1991} for location $M$-estimators.

\subsection{Wilcoxon-signed rank statistic}
Assume that
$(Y_i)_{i\geq 1}$ satisfy $Y_i=\theta+X_i$ where
$(X_i)_{i\geq 1}$ satisfy Assumption (A\ref{assum:long-range}).
The Wilcoxon-signed rank statistic first proposed by
\cite{wilcoxon:1945} can be used to test
the null hypothesis $(H_0)$:  `` $\theta=0$ ''  against the one-sided
alternative $(H_1)$: `` $\theta>0$ '', based on the observations $Y_1,\dots,Y_n$. It
is defined as
$$
T_n=\sum_{j=1}^n R_j \1_{\{X_j>0\}}\; ,
$$
where the $R_j$'s are the ranks of
$X_1,\dots,X_n$. Thus $T_n$ is the sum of the ranks of the positive
observations. Let us study this statistic under the null hypothesis.
One will reject the null hypothesis if the value of $T_n$ is large.
Following \cite{dewan:rao:2005}, $T_n$ can be written as
\begin{equation}\label{e:DPR}
T_n=\sum_{i=1}^n \1_{\{X_i>0\}}+\sum_{1\leq i<j\leq n}\1_{\{X_i+X_j>0\}}=:n
U_{n,1}+\frac{n(n-1)}{2}U_{n,2}\; .
\end{equation}
The Hermite rank of $\1_{\{\cdot>0\}}-\PP(X_1>0)$ equals 1, because
$\PE[X_1(\1_{\{X_1>0\}}-\PP(X_1>0))]>0$. We then deduce
from Theorem 6 of \cite{Arcones:1994} that
\begin{equation}\label{eq:U_1}
n^{D/2} L(n)^{-1/2}(U_{n,1}-\PP(X_1>0))=O_p(1)\; .
\end{equation}
The asymptotic properties of $U_{n,2}$ can be derived from those
of $U_n(0)$ where $U_n(\cdot)$ is the $U$-process defined in (\ref{def:U_n})
with $G(x,y)=x+y$.
Using the results obtained in the study of the Hodges-Lehmann
estimator, we obtain that $\alpha_{1,0}(r)=\alpha_{0,1}(r)=-\varphi(r/\sqrt{2})/\sqrt{2}$,
which is different from 0 for all $r$ in $\rset$ since $\varphi$
is the p.d.f of a standard Gaussian random variable. Thus, the Hermite rank
of the class of functions $\{\1_{G(\cdot,\cdot)\leq r}-\alpha_{0,0}(r); r\in\rset\}$
is equal to 1.
Using the same arguments as those used in the previous example, the assumptions of
Theorem \ref{theo:D<1/2} are fulfilled with $m=1$. Since
$2\alpha_{1,0}(0)=-2\varphi(0)/\sqrt{2}=-1/\sqrt{\pi}$, we get
\begin{equation}\label{eq:U_2}
n^{D/2} L(n)^{-1/2}\left(U_{n,2}-U_2(0)\right)\stackrel{d}{\longrightarrow}
\frac{k(D)^{-1/2}}{\sqrt{\pi}} Z_{1,D}(1)\; ,
\end{equation}
where $U_2(0)=\int\int\1_{\{x+y>0\}}\varphi(x)\varphi(y) \rmd x
\rmd y=1/2$ and $k(D)$ is the constant given in (\ref{e:kD}).
From (\ref{e:DPR}), (\ref{eq:U_1}) and (\ref{eq:U_2}), we get
\begin{multline*}
\frac{2n^{D/2}}{n(n-1)L(n)^{1/2}}(T_n-n \PP(X_1>0)-n(n-1)U_2(0)/2)\\
=\frac{n^{D/2}}{L(n)^{1/2}}\left(U_{n,2}-U_2(0)\right)+o_p(1)
\stackrel{d}{\longrightarrow}\frac{k(D)^{-1/2}}{\sqrt{\pi}} Z_{1,D}(1)\; ,
\end{multline*}
which can be rewritten as follows
\begin{equation}\label{e:limitW}
n^{D/2} L(n)^{-1/2}\left(\frac{2}{n(n-1)}T_n-\frac{1}{n-1}-1/2\right)
\stackrel{d}{\longrightarrow}\frac{k(D)^{-1/2}}{\sqrt{\pi}} Z_{1,D}(1)\; .
\end{equation}
We have thus proved
\begin{prop}
In the long-memory case with $0<D<1$, the asymptotic behavior of the
Wilcoxon-signed rank statistic $T_n$ is Gaussian and given by
(\ref{e:limitW}).
\end{prop}
The Wilcoxon one-sample statistic $U_{n,2}$ was also studied
by \cite{hsing:wu:2004} pp. 1617-1618 by using a different approach.
We obtain the additional constant
$k(D)^{-1/2}$ in the limiting distribution compared to their
result.
\subsection{Sample correlation integral}
In the past few years, a lot of attention has been paid to the
estimation of the correlation dimension of a strange attractor. In many examples,
the correlation dimension $\alpha$ of an invariant probability measure
$\mu$ can be expressed through the \emph{correlation integral} $
C_{\mu}(r)=(\mu\times\mu)\{(x,y):\; |x-y|\leq r\}
$
through
$
C_{\mu}(r)\approx C r^{\alpha},
$
as $r$ tends to 0, where $C$ is a constant.
For further details on the correlation dimension and its applications,
see \cite{borovkova:burton:dehling:2001}.
\cite{grassberger:procaccia:1983} proposed an estimator of the
correlation dimension based on the sample correlation integral
$U_n(r)$, defined in (\ref{def:U_n}) with $G(x,y)=|x-y|$.
In this case,
$\alpha_{1,0}(r)=\alpha_{0,1}(r)=\int_{\rset}x\1_{\{|x-y|\leq r\}}
\varphi(x)\varphi(y)\rmd x\rmd y\\=\int_{\rset} x[\Phi(x+r)-\Phi(x-r)]\varphi(x)\rmd x$,
where, as before, $\varphi$ and $\Phi$ are the p.d.f. and the c.d.f. of a standard
Gaussian random variable, respectively.
Using the symmetry of a standard Gaussian random variable, one gets
$\alpha_{1,0}(r)=\alpha_{0,1}(r)=0$.
Lengthy but straightforward computations lead to
\begin{equation}\label{eq:coeff_corr}
\alpha_{2,0}(r)=\alpha_{0,2}(r)=-\alpha_{1,1}(r)=\dot{\varphi}(r/\sqrt{2})\; ,
\end{equation}
where $\dot{\varphi}$ denotes the first derivative of $\varphi$.
It is non-null if $r\neq 0$.
Thus, for any compact interval $I$ which
does not contain 0, the Hermite rank
of the class of functions $\{\1_{G(\cdot,\cdot)\leq r}-\alpha_{0,0}(r), r\in I\}$
is equal to 2.
Let us assume that $(X_i)_{i\geq 1}$ satisfy Assumption
(A\ref{assum:long-range}). In the case where $D>1/2$, let us check
the assumptions of Theorem \ref{theo:U_n_D>1/2}.
%
%
Conditions (\ref{cond:h_suppl}) and (\ref{cond:contr_G}) can be easily checked and
Condition (\ref{eq:G_lip}) is fulfilled by using similar arguments
as those used in the example of the Hodges-Lehmann estimator.
Conditions (\ref{cond:h_1}) and (\ref{assum:h_1_Lipschitz}) are satisfied since
\begin{equation}\label{eq:h_1:grassberger}
h_1(x,r)=\int_{\rset}\1_{\{|x-y|\leq r\}}\varphi(y)\rmd y=\Phi(x+r)-\Phi(x-r)\; ,
\end{equation}
where $\Phi$ is the
c.d.f of a standard Gaussian random variable.
Thus, in the case where $D>1/2$,
$$\{\sqrt{n}(U_n(r)-U(r)), r\in I\}$$
converges weakly in $\mathcal{D}(I)$, equipped with the
topology of uniform convergence, to the zero mean Gaussian process $\{W(r),r\in I\}$ with
covariance structure given by
\begin{multline}\label{e:limit-covC}
\PE[W(s)W(t)]=4\;\Cov(h_1(X_1,s),h_1(X_1,t))\\
+4\sum_{\ell\geq1}\{\Cov(h_1(X_1,s),h_1(X_{\ell+1},t))+\Cov(h_1(X_1,t),h_1(X_{\ell+1},s))\}\; ,
\end{multline}
where $h_1$ is given in (\ref{eq:h_1:grassberger}).
If $D<1/2$,
with similar arguments as those used in the example on the
Hodges-Lehmann estimator,
the assumptions of Theorem \ref{theo:D<1/2} are satisfied with $m=2$ and we get
using (\ref{eq:coeff_corr}), that
$$
\left\{k(D) n^{D} L(n)^{-1}\left(U_n(r)-U(r)\right); r\in I\right\}
$$
converges weakly in $\mathcal{D}(I)$, equipped with the
topology of uniform convergence, to
\begin{equation}\label{e:limit2-C}
\{\dot{\varphi}(r/\sqrt{2})(Z_{2,D}(1)-Z_{1,D}(1)^2); r\in I\}\; .
\end{equation}
where $I$ is any compact set of $\rset$ which does not contain 0.
Thus
\begin{prop}
In the long-memory case with $1/2<D<1$, the asymptotic behavior of the
sample correlation integral $U_n(r)$, $r\in I$ is Gaussian with
covariance (\ref{e:limit-covC}). If $0<D<1/2$ and if $I$ is a compact
set in $\mathbb{R}$ which does not contain 0, then the limit is
non-Gaussian and given in (\ref{e:limit2-C}).
\end{prop}
\subsection{Shamos scale estimator}
Assume that
$(Y_i)_{i\geq 1}$ satisfy $Y_i=\sigma X_i$ where
$(X_i)_{i\geq 1}$ satisfy Assumption (A\ref{assum:long-range}).
The results of the previous subsection can be used to derive the
properties of the estimator of the scale $\sigma$
proposed by \cite{shamos:1976} and \cite{bickel:lehmann:1979}.
From $Y_1,\dots,Y_n$, it is defined by
$$
\hat{\sigma}_{BL}=c\;\textrm{median}\{|Y_i-Y_j|; 1\leq i<j\leq n\}
=c\;\sigma\;\textrm{median}\{|X_i-X_j|; 1\leq i<j\leq n\}\; ,
$$
where $c=1/(\sqrt{2}\Phi^{-1}(1/4)\approx 1.0483$ and
$\Phi$ is the c.d.f of a standard Gaussian random variable to achieve consistency
for $\sigma$ in the case of Gaussian distribution. $\hat{\sigma}_{BL}$ involves the median
of the distance between observations. As is the case for the
standard deviation, if the $Y_i$'s are transformed into
$aY_i+b$, then $\hat{\sigma}_{BL}$ is multiplied by $|a|$.
Here $G(x,y)=|x-y|$, $U(r)=\int[\Phi(x+r)-\Phi(x-r)]\varphi(x)\rmd x$,
$U'(r)=2\int\varphi(x+r)\varphi(x) \rmd x$, $U^{-1}(1/2)=1/c$
and $U'(U^{-1}(1/2))=U'(1/c)=\sqrt{2}\varphi(1/(c\sqrt{2}))$.
By Corollary \ref{coro:quantile_1},  we obtain that for $D>1/2$,
\begin{equation}\label{e:shamos+demi}
\sqrt{n}(\hat{\sigma}_{BL}-\sigma)\stackrel{d}{\longrightarrow}
-\frac{c\sigma W(1/c)}{\sqrt{2}\varphi(1/(c\sqrt{2}))}\; ,
\end{equation}
where $W$ is a Gaussian process having the covariance structure
(\ref{eq:cov_struct}) with $h_1$ given in (\ref{eq:h_1:grassberger}).
Consider now the case $D<1/2$.
By (\ref{eq:coeff_corr}),
$\alpha_{2,0}(U^{-1}(1/2))=-\alpha_{1,1}(U^{-1}(1/2))=-\alpha_{1,1}(1/c)
=\dot{\varphi}(1/(c\sqrt{2}))$.
Hence, we deduce from Corollary \ref{coro:quantile_2} that, if $D<1/2$,
\begin{multline}\label{e:Sh}
k(D) n^{D} L(n)^{-1}(\hat{\sigma}_{BL}-\sigma)\stackrel{d}{\longrightarrow}
\frac{c\sigma\dot{\varphi}(1/(c\sqrt{2}))}{\sqrt{2}\varphi(1/(c\sqrt{2}))}(Z_{1,D}(1)^2-Z_{2,D}(1))\\
=\frac{\sigma}{2}(Z_{2,D}(1)-Z_{1,D}(1)^2)\; .
\end{multline}
Let us now compare the asymptotic behavior of the Shamos scale
estimator with that of the square root of the sample variance
estimator,
$\hat{\sigma}_{n,Y}=(\sum_{i=1}^n (Y_i-\bar{Y})^2/(n-1))^{1/2}.$
We have
$$
n(n-1)(\hat{\sigma}^2_{n,Y}-\sigma^2)=\sigma^2[n\sum_{i=1}^n
(X_i^2-1)-\sum_{1\leq i,j\leq n} X_i X_j+n]\; ,
$$
so that by Lemma \ref{lem:cumulants},
\begin{equation*}
k(D) n^{D}
L(n)^{-1}(\hat{\sigma}^2_{n,Y}-\sigma^2)\stackrel{d}{\longrightarrow}
\sigma^2 (Z_{2,D}(1)-Z_{1,D}(1)^2)\; .
\end{equation*}
We apply the Delta method to go from $\sigma^2$ to $\sigma$, setting
$f(x)=\sqrt{x}$, so that $f'(\sigma^2)=1/(2\sqrt{\sigma^2})=1/(2\sigma)$.
We obtain
\begin{equation}\label{e:SD}
k(D) n^{D}
L(n)^{-1}(\hat{\sigma}_{n,Y}-\sigma)\stackrel{d}{\longrightarrow}
\frac{\sigma}{2}(Z_{2,D}(1)-Z_{1,D}(1)^2)\; .
\end{equation}
Thus,
\begin{prop}\label{prop:shamos_scale}
In the long-memory case with $1/2<D<1$, the asymptotic behavior of the
Shamos scale estimator $\hat{\sigma}_{BL}$ is Gaussian and given in
(\ref{e:shamos+demi}).
If $0<D<1/2$, it is non-Gaussian and given
by (\ref{e:Sh}); in this case, $\hat{\sigma}_{BL}$  converges to
$\sigma$ at the same rate as the
square root of the sample variance estimator with no loss
of efficiency.
\end{prop}
\section{Proofs of Theorems \ref{theo:U_n_D>1/2} and \ref{theo:D<1/2}}\label{sec:proofs}
\subsection{Lemmas used in the proof of Theorem~\ref{theo:U_n_D>1/2}}\label{subsec:theo1}
\begin{lemma}\label{lem:O_P(1)}
Under the assumptions of Theorem \ref{theo:U_n_D>1/2},
the process $\{\sqrt{n}W_n(r),r\in I\}$, where $W_n(\cdot)$ is defined in (\ref{eq:W_n_hoeff}),
converges weakly in the space of cadlag functions equipped with the
topology of uniform convergence to the zero mean Gaussian process $\{W(r),r\in I\}$ with
covariance structure given by
\begin{multline*}
\PE[W(s)W(t)]=4\; \Cov(h_1(X_1,s),h_1(X_1,t))\\
+4\sum_{\ell\geq1}\{\Cov(h_1(X_1,s),h_1(X_{\ell+1},t))+\Cov(h_1(X_1,t),h_1(X_{\ell+1},s))\}\; .
\end{multline*}
\end{lemma}

The proof of Lemma \ref{lem:O_P(1)} is in the supplemental article \cite{supplement}.

\begin{lemma}\label{lem:reste1}
Under the assumptions of Theorem \ref{theo:U_n_D>1/2},
$$
\sup_{r\in I}\sqrt{n} |R_n(r)|=o_P(1)\; ,
$$
where $R_n$ is defined in (\ref{eq:R_n_hoeff}).
\end{lemma}

The proof of Lemma \ref{lem:reste1} is in the supplemental article \cite{supplement}.

\begin{lemma}\label{eq:maj_Rn}
Under the assumptions of Theorem \ref{theo:U_n_D>1/2}, there exist
positive constants $C$ and $\alpha$ such that, for large
enough $n$,
\begin{equation}\label{e:boundR}
\PE[\{R_n(t)-R_n(s)\}^2]\leq C\; \frac{|t-s|}{n^{1+\alpha}}\; ,
\textrm{ for all } s,t\in I\; ,
\end{equation}
where $R_n$ is defined in (\ref{eq:R_n_hoeff}).
\end{lemma}

The proof of Lemma \ref{eq:maj_Rn} can be found at the end of this
subsection and is based on the following lemmas
\ref{lem:controle_hermite}, \ref{lem:ext:soulier} and \ref{l:bar-gamma}.

\begin{lemma}\label{lem:controle_hermite}
Let $f:\rset^2\longrightarrow \rset$ be a bounded function such that its derivative
$\partial^6 f/\partial x^3\partial y^3$ exists. Let $(X,Y)$ be a
standard Gaussian random vector. Assume that
$\PE\left[ \left(\partial^{i+j} f(X,Y)/\partial x^i\partial y^j \right)^2 \right]<\infty$,
for all $1\leq i,j\leq 3$,
then the Hermite coefficients of $f$ defined by
$c_{p,q}(f):=\PE[f(X,Y) H_p(X) H_q(Y)]$ satisfy, for $p,q\geq 3$
\begin{equation}\label{eq:c_pq:bound}
|c_{p,q}(f)|\leq \PE[(\partial^6 f(X,Y)/\partial x^3\partial
y^3)^2]^{1/2}\sqrt{(p-3)!}\sqrt{(q-3)!}\; .
\end{equation}
\end{lemma}

The proof of Lemma \ref{lem:controle_hermite} is in the supplemental article
\cite{supplement}.
The following lemma is an extension of Corollary 2.1 in
\cite{soulier:2001} and is proved in the supplemental article
\cite{supplement}.

\begin{lemma}\label{lem:ext:soulier}
Let $f_1$ and $f_2$ be two functions defined on $\mathbb{R}^{a_1}$ and
$\mathbb{R}^{a_2}$, respectively. Let $\Gamma$ be the covariance
matrix of the mean-zero Gaussian vector $Y=(Y_1,Y_2)$ where $Y_1$ and $Y_2$ are
in $\mathbb{R}^{a_1}$ and $\mathbb{R}^{a_2}$, respectively.
Assume that there exists a block diagonal matrix  $\Gamma_0$ of size
$(a_1+a_2)\times(a_1+a_2)$ built from $\Gamma$ with diagonal blocks
$\Gamma_{0,1}$ and $\Gamma_{0,2}$ of size $a_1\times a_1$ and
$a_2\times a_2$, respectively, such that
$r^{\star}:=\|\Gamma_0^{-1/2}(\Gamma_0-\Gamma)\Gamma_0^{-1/2}\|_2
\leq (1/3-\varepsilon)$,
for some positive $\varepsilon$. In the previous inequality $\|B\|_2$
denotes the spectral radius of the symmetric matrix $B$.
If at least one function $f_i$ has an Hermite rank larger than $\tau$,
then there exists a positive constant $C(a_1,a_2,\varepsilon)$ such that
\begin{equation}\label{eq:soulier_result}
|\PE[f_1(Y_1) f_2(Y_2)]|\leq C(a_1,a_2,\varepsilon) \|f_1\|_{2,\Gamma_{0,1}} \|f_2\|_{2,\Gamma_{0,2}}
(r^{\star})^{[(\tau+1)/2]}\; ,
\end{equation}
where $\|f_i\|^2_{2,\Gamma_{0,i}}=(2\pi)^{-a_i/2}|\Gamma_{0,i}|^{-1/2}
\int_{\mathbb{R}^{a_i}} f^2_i(x) \exp(-x^T \Gamma^{-1}_{0,i}x/2)
\rmd x$, $i=1,2$ and $[x]$ denotes the integer part of $x$.
\end{lemma}

We shall use the following notation: for a
  Gaussian vector $(X_1,X_2,X_3,X_4)$ with covariance matrix $\Gamma$
  and for any real-valued function of this vector, the expected value
$\PE[f(X_1,X_2,X_3,X_4)]$ will be denoted by
$\PE_{\Gamma}[f(X_1,X_2,X_3,X_4)]$.

\begin{lemma}\label{l:bar-gamma}
Let $(X_1,X_2,X_3,X_4)$ be a Gaussian vector with mean 0 and
covariance matrix
$$
\Gamma=
\begin{bmatrix}
\Gamma_{11}&\Gamma_{12}\\
\Gamma_{21}&\Gamma_{22}\\
\end{bmatrix}
=
\begin{bmatrix}
1&\rho_{12}&\vline&\rho_{13}&\rho_{14}\\
\rho_{12}&1&\vline&\rho_{23}&\rho_{24}\\
\hline
\rho_{13}&\rho_{23}&\vline&1&\rho_{34}\\
\rho_{14}&\rho_{24}&\vline&\rho_{34}&1\\
\end{bmatrix}
$$
and let $J_a$ and $J_b$ be functions from $\rset^2$ to $\rset$ such
that $\PE_{\Gamma}[J_a(X_1,X_2)^2]<\infty$ and 
$\PE_{\Gamma}[J_b(X_3,X_4)^2]<\infty$. Then, there is a Gaussian
vector $(\bar{X}_1,\bar{X}_2,\bar{X}_3,\bar{X}_4)$ with mean 0
and covariance matrix
$$
\bar{\Gamma}=
\begin{bmatrix}
\bar{\Gamma}_{11}&\bar{\Gamma}_{12}\\
\bar{\Gamma}_{21}&\bar{\Gamma}_{22}\\
\end{bmatrix}
=
\begin{bmatrix}
1&0&\vline&\bar{\rho}_{13}&\bar{\rho}_{14}\\
0&1&\vline&\bar{\rho}_{23}&\bar{\rho}_{24}\\
\hline
\bar{\rho}_{13}&\bar{\rho}_{23}&\vline&1&0\\
\bar{\rho}_{14}&\bar{\rho}_{24}&\vline&0&1\\
\end{bmatrix}
$$
with $\bar{\rho}_{13}=\rho_{13}$, 
$\bar{\rho}_{14}=(\rho_{14}-\rho_{13}\rho_{34})/\sqrt{1-\rho_{34}^2}$,
$\bar{\rho}_{23}=(\rho_{23}-\rho_{12}\rho_{13})/\sqrt{1-\rho_{12}^2}$,
\begin{equation*}
\bar{\rho}_{24}=\frac{\rho_{24}+\rho_{12}\rho_{13}\rho_{34}-\rho_{12}\rho_{14}-\rho_{23}\rho_{34}}
{\sqrt{1-\rho_{34}^2}\sqrt{1-\rho_{12}^2}}\; .
\end{equation*}
If $|\rho_{ij}|\leq\rho^\star$ for all $i,j$, then
$\bar{\rho}_{ij}\leq\bar{\rho}^\star$ for all $i,j$, where
\begin{equation}\label{e:rhobar}
\bar{\rho}^\star=\frac{4\rho^{\star}}{1-(\rho^\star)^2}\; .
\end{equation}
There are, moreover, functions $\bar{J}_a$ and $\bar{J}_b$ such that
$$
\PE_{\Gamma}[J_a(X_1,X_2) J_b(X_3,X_4)]
=\PE_{\bar{\Gamma}}[\bar{J}_a(\bar{X}_1,\bar{X}_2)
\bar{J}_b(\bar{X}_3,\bar{X}_4)]\; ,
$$
$$
\PE_{\Gamma}[J_a(X_1,X_2)]=\PE_{\bar{\Gamma}}[\bar{J}_a(\bar{X}_1,\bar{X}_2)]\;
,\;
\PE_{\Gamma}[J_b(X_3,X_4)]=\PE_{\bar{\Gamma}}[\bar{J}_b(\bar{X}_3,\bar{X}_4)]\; ,
$$
and
$$
\PE_{\Gamma}[J_a(X_1,X_2)^2]=\PE_{\bar{\Gamma}}[\bar{J}_a(\bar{X}_1,\bar{X}_2)^2]\;
,\;
\PE_{\Gamma}[J_b(X_3,X_4)^2]=\PE_{\bar{\Gamma}}[\bar{J}_b(\bar{X}_3,\bar{X}_4)^2]\; .
$$
If $J_a$ and $J_b$ are bounded, then $\bar{J}_a$ and $\bar{J}_b$ are
bounded as well.
\end{lemma}

The proof of Lemma \ref{l:bar-gamma} is in the supplemental article
\cite{supplement}.

\begin{proof}[Proof of Lemma \ref{eq:maj_Rn}]
Note that $R_n(t)-R_n(s)$ can be written as
\begin{eqnarray}\label{eq:R_n(t)-R_n(s)}
R_n(t)-R_n(s)=\frac{1}{n(n-1)}\sum_{1\leq i\neq j\leq n} J(X_i,X_j)\; ,
\end{eqnarray}
where
\begin{multline}\label{eq:decomp_J}
J(x,y)=J_{s,t}(x,y)=\{h(x,y,t)-h(x,y,s)\}-\{h_1(x,t)-h_1(x,s)\}\\
-\{h_1(y,t)-h_1(y,s)\}+\{U(t)-U(s)\}.
\end{multline}
In the sequel, we
shall drop for convenience the subscripts $s$ and $t$.
In view of the definition of $h$, $h_1$ and $U$ in
(\ref{def:h}), (\ref{def:h_1}) and (\ref{def:U_hoeff}) respectively,
one has
\begin{equation}\label{e:Jbounded}
\|J\|_{\infty}\leq 4\; ,
\end{equation}
that is, $J$ is bounded. Then, by Conditions (\ref{eq:G_lip})
and  (\ref{assum:h_1_Lipschitz}), for any Gaussian vector
$(X_i,X_j,X_k,X_{\ell})$, one has
\begin{equation}\label{e:J2}
\PE[|J(X_i,X_j)J(X_k,X_{\ell})|]\leq C \;\PE[|J(X_i,X_j)|]\leq C\; |t-s|\; ,
\end{equation}
for some positive constant $C$ which may change from line to line.
By the degeneracy of Hoeffding projections, expanding $J$ into the basis of Hermite polynomials
leads to:
\begin{equation}\label{eq:exp_J}
J(x,y)=\sum_{p,q\geq0}\frac{c_{p,q}(s,t)}{p!\;q!}
H_{p}(x)  H_{q}(y)\;, \text{ with }\;
c_{0,p}=c_{p,0}=0 \; , \forall p\geq 0\; ,
\end{equation}
where
\begin{equation}\label{e:c1}
c_{p,q}(s,t)=\PE[J(X,Y)H_{p}(X)H_{q}(Y)]\; ,
\end{equation}
$X$ and $Y$ being independent standard Gaussian random
variables.
Therefore, using (\ref{e:J2}),
\begin{equation}\label{e:c2}
|c_{p,q}|\leq \PE[J(X,Y)^2]^{1/2} (p!\;q!)^{1/2}\leq C (p!\;q!)^{1/2} |t-s|^{1/2}\; .
\end{equation}
Remark that the sum in (\ref{eq:exp_J}) is over $p$ and $q$ such that
$p+q\geq m$, since the Hermite rank of $J$ is greater than or equal to
the Hermite rank of $h$.
Using (\ref{eq:R_n(t)-R_n(s)}), we obtain that
\begin{equation}\label{e:RJ}
\PE[\{R_n(t)-R_n(s)\}^2]\leq\frac{1}{n^2(n-1)^2}
\sum_{\stackrel{1\leq i_1\neq i_2\leq n}{1\leq i_3\neq i_4\leq n}}
\PE[J(X_{i_1},X_{i_2})J(X_{i_3},X_{i_4})]\; .
\end{equation}
\noindent
We shall consider 3 cases depending on the cardinality of the set
$\{i_1,i_2,i_3,i_4\}$.

1) We first address the case where $i_1=i_3$ and $i_2=i_4$. Using
(\ref{e:J2}), we get
$$
\frac{1}{n^2(n-1)^2}\sum_{1\leq i_1\neq i_2\leq n}
\PE[J(X_{i_1},X_{i_2})^2]\leq\frac{C}{n^2}|t-s|\; ,
$$
which is consistent with (\ref{e:boundR}).

2) Let us now consider the case where the cardinality of the set
$\{i_1,i_2,i_3,i_4\}$ is 3 and suppose without loss of generality that
$i_1=i_3$. Suppose also that $\rho$ defined in Assumption
(A\ref{assum:long-range}) has the following property: there exists
some positive $\rho^\star$ such that
\begin{equation}\label{eq:r_star}
|\rho(k)|\leq \rho^\star<1/13,\textrm{ for all }k\geq 1\;.
\end{equation}
If we apply the same arguments as in the previous case, we get a rate
of order $1/n$ instead of the desired rate $1/n^{1+\alpha}$.
To obtain the latter rate, we propose to approximate $J$ by
a smooth function $J_{\varepsilon}$ using a convolution approach.
More precisely, we define, for all $x,y$ in $\rset$,
\begin{equation}\label{eq:defJ_vareps}
J_{\varepsilon}(x,y)
=\int J(x-\varepsilon z,y-\varepsilon
z')\varphi(z)\varphi(z')\rmd z \rmd z'\; .
\end{equation}
Thus,
\begin{multline}\label{e:JJ}
\PE[J(X_{i_1},X_{i_2})J(X_{i_1},X_{i_4})]=
\PE[J_{\varepsilon}(X_{i_1},X_{i_2})J(X_{i_1},X_{i_4})]\\
+\PE[(J-J_{\varepsilon})(X_{i_1},X_{i_2})J(X_{i_1},X_{i_4})]\; .
\end{multline}
Applying Lemma \ref{lem:controle_hermite} to $f=J_{\varepsilon}$ and
noting that, by Condition (\ref{cond:h_suppl}), $\|\partial^6 J_{\varepsilon} /\partial x^3\partial
y^3\|\leq C \varepsilon^{-6}|t-s|^{1/2}$, for some positive constant $C$, we
obtain
\begin{multline*}
\PE[J_{\varepsilon}(X_{i_1},X_{i_2})J(X_{i_1},X_{i_4})]\\
\leq \PE[\sum_{p,q\geq1}\frac{|c_{p,q}(J_\varepsilon)|}{p!\;q!}
|H_{p}(X_{i_1})J(X_{i_1},X_{i_4})H_{q}(X_{i_2})|]\\\leq C
\varepsilon^{-6}|t-s|^{1/2}\sum_{p,q\geq 3}(p! q!)^{-1} \sqrt{(p-3)!} \sqrt{(q-3)!}\\
 |\PE[H_{p}(X_{i_1}) J(X_{i_1},X_{i_4})H_{q}(X_{i_2})]|\; ,
\end{multline*}
where $c_{p,q}(J_\varepsilon)$ is the $(p,q)$th Hermite coefficient
of $J_\varepsilon$.
We shall apply Lemma \ref{lem:ext:soulier}
with $Y_1=(X_{i_1},X_{i_4})$, $Y_2=X_{i_2}$, $a_1=2$, $a_2=1$, $\Gamma_{0,1}=\textrm{Id}$, $\Gamma_{0,2}=1$,
$f_1=H_p J$ and $f_2=H_q$. Observe that $\textrm{Id}-\Gamma$ is a
$3\times 3$ matrix with $\rho$ entries and hence
$\|\textrm{Id}-\Gamma\|_2\leq (a_1+a_2)\|\textrm{Id}-\Gamma\|_{\infty}=3\|\textrm{Id}-\Gamma\|_{\infty}=3\rho^{\star}$,
where $\|A\|_{\infty}$ is defined for a matrix $A=(a_{i,j})_{i,j}$
by $\|A\|_{\infty}=\max_{i,j} |a_{i,j}|$. Hence, by (\ref{eq:r_star}),
the condition on $r^{\star}$ of Lemma
\ref{lem:ext:soulier} is satisfied. Since $J$ is bounded
and $f_2$ is of Hermite rank larger than 2, Lemma
\ref{lem:ext:soulier} with $[(2+1)/2]=1$ implies that there exists a positive constant $C$ such that
\begin{equation}\label{eq:soulier}
|\PE[H_{p}(X_{i_1}) J(X_{i_1},X_{i_4})H_{q}(X_{i_2})]|
\leq C\sqrt{p!\;q!}\;(|\rho(i_4-i_2)|\vee |\rho(i_2-i_1)|\vee |\rho(i_4-i_1)|)\; .
\end{equation}
Hence,
\begin{equation}
\PE[J_{\varepsilon}(X_{i_1},X_{i_2})J(X_{i_1},X_{i_4})]
\leq C \varepsilon^{-6}|t-s|^{1/2}\;(|\rho(i_4-i_2)|+ |\rho(i_2-i_1)|+ |\rho(i_4-i_1)|)\; .
\end{equation}
Since, for example, $\sum_{1\leq i_2\neq i_4\leq n}|\rho(i_4-i_2)|\leq
n\sum_{|k|<n}|\rho(k)|$, and since there exist positive constants $C$ and $\delta$ such that
$|\rho(k)|\leq C(1\wedge |k|^{-D+\delta})$, for all $k\geq 1$, we obtain
that $\sum_{1\leq i_2\neq i_4\leq n}|\rho(i_4-i_2)|\leq C
n^{2-D+\delta}$. Hence,
\begin{equation}\label{eq:regul}
\frac{1}{n^2(n-1)^2}
\sum_{\stackrel{1\leq i_1\neq i_2\leq n}{1\leq i_1\neq i_4\leq n}}
\PE[J_{\varepsilon}(X_{i_1},X_{i_2})J(X_{i_1},X_{i_4})]
\leq \frac{C \varepsilon^{-6}|t-s|^{1/2}}{n^{1+D-\delta}}\; .
\end{equation}
We now focus on the last term in (\ref{e:JJ}).
By the Cauchy-Schwarz inequality
and (\ref{e:J2}),
\begin{multline}\label{e:Jeps1}
\frac{1}{n^2(n-1)^2}\sum_{\stackrel{1\leq i_1\neq i_2\leq n}{1\leq i_1\neq i_4\leq n}}
\PE[(J- J_{\varepsilon})(X_{i_1},X_{i_2})J(X_{i_1},X_{i_4})]\\
\leq C\frac{|t-s|^{1/2}}{n^2(n-1)}\sum_{1\leq i_1\neq i_2\leq n}
\PE[(J- J_{\varepsilon})^2(X_{i_1},X_{i_2})]^{1/2}\; .
\end{multline}
Using (\ref{eq:defJ_vareps}), the Jensen's inequality and (\ref{e:Jbounded}),
\begin{multline}\label{e:Jeps2}
\PE[(J- J_{\varepsilon})^2(X_{i_1},X_{i_2})]\\
=\int_{\mathbb{R}^2}\left\{\int_{\mathbb{R}^2}[J(x,y)-J(x-\varepsilon z,y-\varepsilon z')]
\varphi(z)\varphi(z')\rmd z\rmd z'\right\}^2 f_{i_1,i_2}(x,y)\rmd x\rmd y\\
\leq\int_{\mathbb{R}^2}\left\{\int_{\mathbb{R}^2}[J(x,y)-J(x-\varepsilon z,y-\varepsilon
z')]^2f_{i_1,i_2}(x,y)\rmd x\rmd y\right\}
\varphi(z)\varphi(z')\rmd z\rmd z'\\
\leq C \int_{\mathbb{R}^2}\{\int_{\mathbb{R}^2}|J(x,y)-J(x-\varepsilon z,y-\varepsilon
z')|f_{i_1,i_2}(x,y)\rmd x\rmd y\}
\varphi(z)\varphi(z')\rmd z\rmd z'\; ,
\end{multline}
where $f_{i_1,i_2}$ is the p.d.f of $(X_{i_1},X_{i_2})$.
By (\ref{eq:decomp_J}), Conditions \eqref{cond:contr_G} and (\ref{cond:h_1}),
\begin{multline}\label{e:Jeps3}
\int_{\mathbb{R}^2}\{\int_{\mathbb{R}^2}|J(x,y)-J(x-\varepsilon z,y-\varepsilon
z')|f_{i_1,i_2}(x,y)\rmd x\rmd y\}
\varphi(z)\varphi(z')\rmd z\rmd z'\\
\leq C\varepsilon\int_{\mathbb{R}^2}(|z|+|z'|)\varphi(z)\varphi(z')\rmd z\rmd z'
\leq C\varepsilon\; .
\end{multline}
Using (\ref{e:Jeps1}), (\ref{e:Jeps2}) and (\ref{e:Jeps3}), we get
\begin{equation}\label{eq:regul_diff}
\frac{1}{n^2(n-1)^2}\sum_{\stackrel{1\leq i_1\neq i_2\leq n}{1\leq i_1\neq i_4\leq n}}
\PE[(J- J_{\varepsilon})(X_{i_1},X_{i_2})J(X_{i_1},X_{i_4})]
\leq C\frac{\varepsilon^{1/2} |t-s|^{1/2}}{n}\; .
\end{equation}
Note that (\ref{eq:regul}) involves the factor $\varepsilon^{-6}$ and
(\ref{eq:regul_diff}) involves the factor $\varepsilon^{1/2}$.
By choosing $\varepsilon=\varepsilon_n=n^{-\nu}$ with
$0<\nu<(D-\delta)/6$ in (\ref{eq:regul}) and (\ref{eq:regul_diff}),
we obtain a result consistent with (\ref{e:boundR}).

If Condition \eqref{eq:r_star} is not satisfied then let
$\tau$ be such that $\rho(k)\leq \rho^\star<1/13$, for all $k>\tau$.
In the case where, for instance, $|i_2-i_4|\leq\tau$ then, using that
$J$ is bounded, Conditions (\ref{eq:G_lip}) and
\eqref{assum:h_1_Lipschitz}, we get that
$$
\frac{1}{n^2(n-1)^2}
\sum_{\stackrel{1\leq i_1\neq i_2\leq n}{1\leq i_1\neq i_4\leq n,|i_2-i_4|\leq\tau}}
\PE[J(X_{i_1},X_{i_2})J(X_{i_1},X_{i_4})]\leq C\frac{\tau |t-s|}{n^2}\; ,
$$
instead of (\ref{eq:regul}), but the result is still consistent with
(\ref{e:boundR}).
The same result holds when $|i_1-i_4|\leq\tau$ or $|i_1-i_2|\leq\tau$.
Note also that the remaining sum over the indices such that $|i_1-i_2|>\tau$,
$|i_1-i_4|>\tau$ and $|i_2-i_4|>\tau$ can be addressed in the same way as
when Condition \eqref{eq:r_star} is satisfied.

3) Now, we assume that the cardinal number of the set
$\{i_1,i_2,i_3,i_4\}$
equals 4 and that Condition \eqref{eq:r_star} holds.
By Lemma \ref{l:bar-gamma},
\begin{equation}\label{e:JJbar}
\PE[J(X_{i_1},X_{i_2})J(X_{i_3},X_{i_4})]
=\PE_{\bar{\Gamma}}[J_{i_1,i_2}(\bar{X}_{i_1},\bar{X}_{i_2})J_{i_3,i_4}(\bar{X}_{i_3},\bar{X}_{i_4})]\;.
\end{equation}
Here $(\bar{X}_{i_1},\bar{X}_{i_2},\bar{X}_{i_3},\bar{X}_{i_4})$ is a
Gaussian vector with mean 0 and covariance matrix $\bar{\Gamma}$
defined in Lemma \ref{l:bar-gamma} where $\rho_{ij}=\rho(i-j)$,
$J_a=J_b=J$,
$\bar{J}_a=J_{i_1,i_2}$ and $\bar{J}_b=J_{i_3,i_4}$.
Since the covariance of $(\bar{X}_{i_1},\bar{X}_{i_2})$ and
$(\bar{X}_{i_3},\bar{X}_{i_4})$ is the identity matrix, we can expand
$J_{i_1,i_2}(\bar{X}_{i_1},\bar{X}_{i_2})$ and 
$J_{i_3,i_4}(\bar{X}_{i_3},\bar{X}_{i_4})$.
$J_{i_1,i_2}(\bar{X}_{i_1},\bar{X}_{i_2})$
is the limit in $L^2$, as $K\to\infty$, of
\begin{equation}\label{e:JKsum}
J_{i_1,i_2}^K(\bar{X}_{i_1},\bar{X}_{i_2})
=\sum_{p=1}^K \frac{c_{p_1,p_2}^{i_1,i_2}}{\sqrt{p_1! p_2!}}
H_{p_1}(\bar{X}_{i_1})H_{p_2}(\bar{X}_{i_2})\; ,
\end{equation}
with a similar expansion for
$J_{i_3,i_4}^K(\bar{X}_{i_3},\bar{X}_{i_4})$.
Therefore,
\begin{equation}\label{e:JJK}
\lim_{K\to\infty}\PE_{\bar{\Gamma}}[J_{i_1,i_2}(\bar{X}_{i_1},\bar{X}_{i_2})
J_{i_3,i_4}(\bar{X}_{i_3},\bar{X}_{i_4})
-J_{i_1,i_2}^K(\bar{X}_{i_1},\bar{X}_{i_2})
J_{i_3,i_4}^K(\bar{X}_{i_3},\bar{X}_{i_4})]=0\; .
\end{equation}
Thus it is enough to majorize
\begin{multline}\label{e:bd1}
\PE_{\bar{\Gamma}}[J_{i_1,i_2}^K(\bar{X}_{i_1},\bar{X}_{i_2})
J_{i_3,i_4}^K(\bar{X}_{i_3},\bar{X}_{i_4})]\\
\leq \sum_{1\leq p_1,p_2\leq K}\;\sum_{1\leq p_3,p_4\leq K}
\frac{|c_{p_1,p_2}^{i_1,i_2}|}{p_1! p_2!}\frac{|c_{p_3,p_4}^{i_3,i_4}|}{p_3! p_4!}
|\PE_{\bar{\Gamma}}[H_{p_1}(\bar{X}_1) H_{p_2}(\bar{X}_2) H_{p_3}(\bar{X}_3) H_{p_4}(\bar{X}_4)]|\; .
\end{multline}
By Lemma 3.2 in \cite{taqqu:1977}, $\PE_{\bar{\Gamma}}[H_{p_1}(\bar{X}_{i_1})
H_{p_2}(\bar{X}_{i_2}) H_{p_3}(\bar{X}_{i_3}) H_{p_4}(\bar{X}_{i_4})]$
is zero if $p_1+\dots+p_4$ is odd. Otherwise it is bounded by a
constant times a sum of products of $(p_1+\dots+p_4)/2$
covariances. These will be denoted
$\bar{\rho}_{i,j}=\PE(\bar{X}_i\bar{X}_j)$ and are given in Lemma
\ref{l:bar-gamma}. Since $\rho(k)\leq \rho^{\star}<1/13$, we have that
$\bar{\rho}_{i,j}\leq\bar{\rho}^{\star}<1/3$, where 
$\bar{\rho}^{\star}=4\rho^{\star}/(1-(\rho^{\star})^2)$ by (\ref{e:rhobar}).
Bounding, in each product of covariances, all the covariances but two, by
$\bar{\rho}^{\star}<1/3$, we get that \\$\PE_{\bar{\Gamma}}[H_{p_1}(\bar{X}_{i_1})
H_{p_2}(\bar{X}_{i_2}) H_{p_3}(\bar{X}_{i_3}) H_{p_4}(\bar{X}_{i_4})]$ is bounded by
\begin{equation}\label{e:bd2}
C\;(3\bar{\rho}^{\star})^{\frac{p_1+p_2+p_3+p_4}{2}-2}\bar{A}(i_1,i_2,i_3,i_4)
|\PE[H_{p_1}(X) H_{p_2}(X) H_{p_3}(X) H_{p_4}(X)]|\;,
\end{equation}
where, since $\bar{\rho}_{i_1,i_2}=\bar{\rho}_{i_3,i_4}=0$,
\begin{multline}\label{e:Abar}
\bar{A}(i_1,i_2,i_3,i_4)=
\bar{\rho}_{i_1,i_3}\bar{\rho}_{i_2,i_4}
+\bar{\rho}_{i_2,i_3}\bar{\rho}_{i_1,i_4}
+\bar{\rho}_{i_1,i_3}\bar{\rho}_{i_2,i_3}
+\bar{\rho}_{i_1,i_4}\bar{\rho}_{i_2,i_4}\\
+\bar{\rho}_{i_1,i_3}\bar{\rho}_{i_1,i_4}
+\bar{\rho}_{i_2,i_3}\bar{\rho}_{i_2,i_4}\; ,
\end{multline}
and where $X$ is a standard Gaussian random variable.
Note also that the hypercontractivity Lemma 3.1 in \cite{taqqu:1977} yields
\begin{equation}\label{eq:taqqu_bound}
|\PE[H_{p_1}(X)H_{p_2}(X)H_{p_3}(X)H_{p_4}(X)]|\leq 3^{\frac{p_1+p_2+p_3+p_4}{2}}\sqrt{p_1!\;p_2!\;p_3!\;p_4!}\; .
\end{equation}
Thus (\ref{e:bd1}) is bounded by
$$
C\bar{A}\left(\sum_{1\leq p_1,p_2\leq K}
\frac{|c_{p_1,p_2}^{i_1,i_2}|}{\sqrt{p_1!\;p_2!}}(3\bar{\rho}^\star)^{\frac{p_1+p_2}{2}-1}\right)
\left(\sum_{1\leq p_3,p_4\leq K}
\frac{|c_{p_3,p_4}^{i_3,i_4}|}{\sqrt{p_3!\;p_4!}}(3\bar{\rho}^\star)^{\frac{p_3+p_4}{2}-1}\right)\; .
$$
By the Cauchy-Schwarz inequality, the first term in brackets is
bounded by
\begin{multline}\label{e:Abd}
\left(\sum_{1\leq p_1,p_2\leq K}\frac{(c_{p_1,p_2}^{i_1,i_2})^2}{p_1!\;p_2!}\right)^{1/2}
\left(\sum_{1\leq p_1,p_2\leq K}
  (3\bar{\rho}^\star)^{p_1+p_2-2}\right)^{1/2}\\
\leq
\PE_{\textrm{I}}\left[J_{i_1,i_2}(\bar{X}_{i_1},\bar{X}_{i_2})^2\right]^{1/2}
\left(\sum_{p\geq 1}(3\bar{\rho}^\star)^{p-1}\right)\; ,
\end{multline}
where $\textrm{I}$ is the identity matrix and similarly for the second
term. Since $\bar{\rho}^\star<1/3$, it
follows from Lemma \ref{l:bar-gamma} that (\ref{e:bd1}) is bounded by
\begin{multline}\label{e:bd2}
C \bar{A}\;
\PE_{\textrm{I}}\left[J_{i_1,i_2}(\bar{X}_{i_1},\bar{X}_{i_2})\right]^{1/2}
\PE_{\textrm{I}}\left[J_{i_3,i_4}(\bar{X}_{i_3},\bar{X}_{i_4})\right]^{1/2}\\
=C \bar{A}\;
\PE_{\Gamma_{11}}\left[J(X_{i_1},X_{i_2})^2\right]^{1/2}
\PE_{\Gamma_{22}}\left[J(X_{i_3},X_{i_4})^2\right]^{1/2}
\leq C \bar{A} |t-s|\; ,
\end{multline}
where we used (\ref{e:J2}) and the fact that $J$ is bounded.
Thus, in wiew of (\ref{e:JJbar}), (\ref{e:JJK}), (\ref{e:bd1}) and
(\ref{e:bd2}),
we have
\begin{multline}\label{e:bd3}
\frac{1}{n^2(n-1)^2}\sum_{\stackrel{1\leq i_1,i_2,i_3,i_4\leq n}
{|\{i_1,i_2,i_3,i_4\}|=4}}\PE[J(X_{i_1},X_{i_2})J(X_{i_3},X_{i_4})]\\
\leq C\frac{|t-s|}{n^2(n-1)^2}\sum_{\stackrel{1\leq i_1,i_2,i_3,i_4\leq n}
{|\{i_1,i_2,i_3,i_4\}|=4}}\bar{A}(i_1,i_2,i_3,i_4)\; .
\end{multline}
We need to evaluate that sum. Recall that
$\bar{A}=\bar{A}(i_1,i_2,i_3,i_4)$ is defined in (\ref{e:Abar}) with
the $\bar{\rho}_{i,j}$ defined in Lemma \ref{l:bar-gamma}. We shall
treat one summand in $\bar{A}$ (the others are treated in the same
way). We have
$$
\bar{\rho}_{i_1,i_3}\bar{\rho}_{i_2,i_4}
\leq C \rho(i_1-i_3)[\rho(i_3-i_4)+\rho(i_1-i_4)+\rho(i_3-i_4)+\rho(i_2-i_4)]\;.
$$
Using that there exist positive constants $C$ (changing from line to
line) and $\varepsilon$ such that
$|\rho(k)|\leq C(1\wedge |k|^{-D+\varepsilon})$, for all $k\geq 1$, we get
\begin{multline*}
\sum_{\stackrel{1\leq i_1,i_2,i_3,i_4\leq n}{|\{i_1,i_2,i_3,i_4\}|=4}}
\rho(i_1-i_2)\rho(i_3-i_4)=(\sum_{1\leq i_1\neq i_2\leq
  n}\rho(i_1-i_2))^2\leq n^2 (\sum_{|k|<n}\rho(k))^2\\
\leq  C
n^{4-2D+2\varepsilon}\; ,
\end{multline*}
and
\begin{multline*}
\sum_{\stackrel{1\leq i_1,i_2,i_3,i_4\leq n}{|\{i_1,i_2,i_3,i_4\}|=4}}
\rho(i_1-i_2)\rho(i_2-i_4)\leq n\sum_{i_2=1}^n(\sum_{1\leq i_1\neq i_2\leq
  n}
\rho(i_1-i_2))^2\\
=n\sum_{i_2=1}^n(\sum_{i_1=1}^{i_2-1}\rho(i_1-i_2)
+\sum_{i_1=i_2+1}^{n}\rho(i_1-i_2))^2
\leq C n\sum_{i_2=1}^n i_2^{2-2D+2\varepsilon}
\leq C n^{4-2D+2\varepsilon}\; .
\end{multline*}
Therefore, Relation \eqref{e:bd3} is bounded by
$Cn^{-2D+2\varepsilon}|t-s|$, which is a result consistent with
(\ref{e:boundR}) with $2\varepsilon=D-1/2 >0.$

If Condition \eqref{eq:r_star} is not satisfied, then let
$\tau$ be such that
\begin{equation}\label{eq:tau_soulier}
\sup_{k>\tau}\rho(k)\leq \frac{1}{13}(1-\sup_{1\leq
  k\leq\tau}\rho(k))\; .
\end{equation}
In the case where, for instance, $|i_1-i_3|\leq\tau$ and
$|i_2-i_4|\leq\tau$, there is no need to use (\ref{eq:tau_soulier}) because we get using
(\ref{e:J2}),
\begin{multline}
\frac{1}{n^2(n-1)^2}
\sum_{\stackrel{1\leq i_1,i_2,i_3,i_4\leq n}{|\{i_1,i_2,i_3,i_4\}|=4,|i_1-i_3|\leq\tau,|i_2-i_4|\leq\tau}}
\PE[J(X_{i_1},X_{i_2})J(X_{i_3},X_{i_4})]\\
\leq C\frac{\tau^2 |t-s|}{n^2}\; ,
\end{multline}
which is consistent with (\ref{e:boundR}).
In the case where, for instance, $|i_1-i_3|\leq\tau$ and the other
distances are larger than $\tau$, we apply the
same method as in 2). What changes
is the block diagonal matrix $\Gamma_0$ involved in Lemma \ref{lem:ext:soulier}.
In fact, to evaluate $\PE[J_{\varepsilon}(X_{i_1},X_{i_2})J(X_{i_3},X_{i_4})]$
we expand $J_{\varepsilon}$ in Hermite polynomials, so that we need to
control $\PE[H_p(X_{i_1})J(X_{i_3},X_{i_4})H_q(X_{i_2})]$.
We want to apply Lemma \ref{lem:ext:soulier} with
$Y_1=(X_{i_1},X_{i_3},X_{i_4})$, $Y_2=X_{i_2}$, $f_1=H_p J$, $f_2=H_q$.
We let  $\Gamma_{0,1}$ be a $3\times 3$ block diagonal matrix with a
first block corresponding to the covariance matrix of the vector
$(X_{i_1},X_{i_3})$ and a second block equal to 1, and we let
$\Gamma_{0,2}=1$, so that $\Gamma_0$ is a $4\times 4$ matrix.
Observe that $\|\Gamma_0^{-1}(\Gamma-\Gamma_0)\|_2
\leq 4\|\Gamma_0^{-1}(\Gamma-\Gamma_0)\|_{\infty}$, where
$\|\Gamma_0^{-1}(\Gamma-\Gamma_0)\|_{\infty}\leq
(\sup_{1\leq k\leq\tau}\rho(k))(1-\sup_{k>\tau}\rho(k))^{-1}\leq 1/13$,
by (\ref{eq:tau_soulier}).
Thus,  $\|\Gamma_0^{-1}(\Gamma-\Gamma_0)\|_2\leq 4/13<1/3-\eta$,
for some positive $\eta$.
Hence, the condition on $r^{\star}$ of Lemma \ref{lem:ext:soulier} is satisfied.
The remaining sum over indices where the
distances between any two indices are larger than $\tau$ can be addressed in the same way as
when Condition \eqref{eq:r_star} is satisfied.
\end{proof}

\subsection{Lemmas used in the proof of Theorem \ref{theo:D<1/2}} \label{subsec:theo2:lem}
\ \\
The following lemma proves joint convergence and provides the joint
cumulants of the limits $(Z_{2,D}(1),(Z_{1,D}(1))^2).$
\begin{lemma}\label{lem:cumulants}
Let $(X_j)_{j\geq 1}$ be a stationary process satisfying Assumption
(A\ref{assum:long-range}) with $D<1/2$ and let $a$ and $b$ be two
real constants. Then, as $n$ tends to infinity,
$$
k(D)\frac{n^{D-2}}{L(n)}\left\{a n \sum_{i=1}^n (X_i^2-1)+b\sum_{1\leq i,j\leq n} X_i
X_j\right\}\stackrel{d}{\longrightarrow}\left[a Z_{2,D}(1)+b (Z_{1,D}(1))^2\right]\; ,
$$
where $\stackrel{d}{\longrightarrow}$ denotes the convergence in distribution,
$k(D)=\emph{\B}((1-D)/2,D)$ where $\emph{\B}$ denotes the Beta function,
$Z_{1,D}(\cdot)$ and $Z_{2,D}(\cdot)$ are defined in (\ref{eq:fBm}) and
(\ref{eq:rosenblatt}) respectively. The cumulants of the limit process
are given in (\ref{eq:cumulants}).
\end{lemma}

The proof of Lemma \ref{lem:cumulants} is in the supplemental article
\cite{supplement}.

\begin{remark}\label{r:moments}
It follows from Lemma \ref{lem:cumulants} that
$\PE[Z_{D,1}(1)^2]=\sigma^2$ where $\sigma^2$ is given in
(\ref{e:VarFBM}). Moreover, setting $a=1$, $b=0$ in
(\ref{eq:cumulants}), we get the expression (\ref{e:VarRos}) for
$\PE[Z_{D,2}(1)^2]$.
\end{remark}
\begin{lemma}\label{lem:partie_qui_compte}
Under Assumption (A\ref{assum:long-range}) there exists a positive constant $C$ such that, for $n$ large
enough,
\begin{equation}\label{case_(i)}
\frac{n^{D-2}}{L(n)}\PE[\{\sum_{i=1}^n X_i\}^2]
\leq C\; , \textrm{ when } D<1\; ,
\end{equation}
\begin{equation}\label{case_(ii)_a}
\frac{n^{2D-2}}{L(n)^2}\PE[\{\sum_{i=1}^n (X_i^2-1)\}^2]
\leq C \; , \textrm{ when } D<1/2\; ,
\end{equation}
and
\begin{equation}\label{case_(ii)_b}
\frac{n^{2D-4}}{L(n)^2}\PE[\{\sum_{1\leq i\neq j\leq n}X_i X_j\}^2]
\leq C \; , \textrm{ when } D<1/2\; .
\end{equation}
\end{lemma}

The proof of Lemma \ref{lem:partie_qui_compte} is in the supplemental article \cite{supplement}.

\begin{lemma}\label{lem:reste}
Suppose that the assumptions of Theorem \ref{theo:D<1/2} hold, in
particular $D<1/m$ where $m=1$ or 2. Then, $\widetilde{R}_n$ defined in
(\ref{eq:herm_decomp_Un}) satisfies the following property. There exist positive
constants $\alpha$ and $C$ such that,
for $n$ large enough,
\begin{equation}\label{e:Rtilde}
a_n^2\; \PE[(\widetilde{R}_{n}(t)-\widetilde{R}_{n}(s))^2]\leq C \frac{|t-s|}{n^{\alpha}}\;,\textrm{
  for all } s,t\in I\; ,
\end{equation}
where $I$ is any compact interval of $\rset$ and $a_n=n^{mD/2-2}L(n)^{-m/2}$.
\end{lemma}

The proof of Lemma \ref{lem:reste} is in the supplemental article \cite{supplement}.

\begin{lemma}\label{lem:reste_D<1/2}
Under the assumptions of Theorem \ref{theo:D<1/2}, $\widetilde{R}_n$
defined in (\ref{eq:herm_decomp_Un}) satisfies, as $n$ tends to infinity,
$$
\sup_{r\in I} a_n |\widetilde{R}_n(r)|=o_P(1)\; ,
$$
where $a_n=n^{-2+mD/2}L(n)^{-m/2}$,  and $m$ is
the Hermite rank of the class of functions $\{h(\cdot,\cdot,r)-U(r)\; ,
r\in I\}$ which is equal to 1 or 2.
\end{lemma}

The proof of Lemma \ref{lem:reste_D<1/2} is in the supplemental article
\cite{supplement}.

\subsection{Proof of  Theorem \ref{theo:D<1/2}}\label{subsec:theo2:proof}
Consider the decomposition (\ref{eq:herm_decomp_Un}).
Since $G(x,y)=G(y,x)$, one has
$\alpha_{1,0}(r)=\alpha_{0,1}(r)$, $\alpha_{2,0}(r)=\alpha_{0,2}(r)$ and
$\widetilde{W}_n$ defined in (\ref{def:W_n}) satisfies
\begin{equation}\label{eq:W_n:m=1}
\widetilde{W}_n(r)=2(n-1)\alpha_{1,0}(r)\sum_{i=1}^n X_i\;,\textrm{ if } m=1\; ,
\end{equation}
\begin{equation}\label{eq:W_n:m=2}
\widetilde{W}_n(r)=\alpha_{1,1}(r)\sum_{1\leq i\neq j\leq n} X_i
X_j+(n-1)\alpha_{2,0}(r)\sum_{i=1}^n (X_i^2-1)\; ,\textrm{ if } m=2\; .
\end{equation}
If $m=1$, using Lemma 5.1 in \cite{taqqu:1975}, if $r$ is fixed, $n^{D/2-2} L(n)^{-1/2}\widetilde{W}_n(r)$
defined in (\ref{eq:W_n:m=1}) converges in distribution to
$2k(D)^{-1/2}\alpha_{1,0}(r) Z_{1,D}(1)$.
Then, by the Cramer-Wold device, if $r_1,\dots,r_k$ are fixed real
numbers,\\ $k(D)^{1/2}n^{D/2-2} L(n)^{-1/2} (\widetilde{W}_n(r_1),\dots,\widetilde{W}_n(r_k))$ converges in
distribution to \\$(2\alpha_{1,0}(r_1)Z_{1,D}(1),\dots,
2\alpha_{1,0}(r_k)Z_{1,D}(1)).$
In the same way, if $m=2$, using Lemma \ref{lem:cumulants} in Section
\ref{subsec:theo2:lem} and
the Cramer-Wold device, \\$k(D)n^{D-2}L(n)^{-1}(\widetilde{W}_n(r_1),\dots,\widetilde{W}_n(r_k))$
converges in distribution to\\
$(\alpha_{1,1}(r_1)(Z_{1,D}(1))^2+
\alpha_{2,0}(r_1)Z_{2,D}(1),\dots,\alpha_{1,1}(r_k)(Z_{1,D}(1))^2+\alpha_{2,0}(r_k)Z_{2,D}(1))$.
We now show that  $\{n^{mD/2-2}L(n)^{-m/2} \widetilde{W}_n(r);  r\in
I\}$ is tight in $\mathcal{D}(I)$.
We shall do it in the case $m=1$. By (\ref{e:WL}),
Lemma \ref{lem:partie_qui_compte} in Section
\ref{subsec:theo2:lem} and the fact that $\widetilde{\Lambda}$ is a bounded
Lipschitz function, we get that there exists a positive constant $C$
such that for all $r_1<r_2$ in $I$,
\begin{multline*}
(n^{D/2-2}L(n)^{-1/2})^2\PE[\{\widetilde{W}_n(r_2)-\widetilde{W}_n(r_1)\}^2]
\leq C (\widetilde{\Lambda}(r_2)-\widetilde{\Lambda}(r_1))^2\leq C
|r_2-r_1|^2\; .
\end{multline*}
Using the Cauchy-Schwarz inequality, we obtain that for all $r_1$,
$r_2$, $r_3$ in $I$, such that $r_1<r_2<r_3$,
\begin{multline*}
(n^{D/2-2}L(n)^{-1/2})^2\PE\left[\left|\widetilde{W}_n(r_2)-\widetilde{W}_n(r_1)\right|\;
\left|\widetilde{W}_n(r_3)-\widetilde{W}_n(r_2)\right|\right]\\
\leq C |r_2-r_1| |r_3-r_2|\leq C |r_3-r_1|^2\; .
\end{multline*}
The tightness then follows from Theorem 15.6 of
\cite{billingsley:1968}.
A similar argument holds for $m=2$.
Thus, $\{n^{mD/2-2}L(n)^{-m/2}
\widetilde{W}_n(r);  r\in I\}$ converges
weakly to
$
\{2\alpha_{1,0}(r)k(D)^{-1/2}Z_{1,D}(1); r\in I\},
$
if $m=1$ and to
$
\{k(D)^{-1}\left[\alpha_{1,1}(r)Z_{1,D}(1)^2+\alpha_{2,0}(r)Z_{2,D}(1)\right]; r\in I\},
$
if $m=2$.
To complete the proof of Theorem \ref{theo:D<1/2} use
(\ref{eq:herm_decomp_Un}) and Lemma
\ref{lem:reste_D<1/2} in Section \ref{subsec:theo2:lem}, which ensures that
$\sup_{r\in I}n^{mD/2-2}L(n)^{-m/2} |\widetilde{R}_n(r)|=o_P(1)$, as $n$
tends to infinity.

\begin{supplement}[id=suppA]
  \stitle{Proofs of Lemmas \ref{lem:O_P(1)}, \ref{lem:reste1}, 
    \ref{lem:controle_hermite}, \ref{lem:ext:soulier}, \ref{l:bar-gamma}, 
    \ref{lem:cumulants}, \ref{lem:partie_qui_compte}, \ref{lem:reste} and 
\ref{lem:reste_D<1/2} and some numerical experiments.}
  \slink[url]{http://lib.stat.cmu.edu/aoas/???/???}
  \sdescription{This supplement contains proofs of Lemmas
\ref{lem:O_P(1)}, \ref{lem:reste1},
    \ref{lem:controle_hermite}, \ref{lem:ext:soulier}, \ref{l:bar-gamma},  
    \ref{lem:cumulants}, \ref{lem:partie_qui_compte}, \ref{lem:reste} and 
\ref{lem:reste_D<1/2} and a section containing numerical experiments
illustrating some results of Section \ref{sec:appli}.}
\end{supplement}

\bibliography{robust_cov_biblio}

\begin{center}
\textbf{Supplement to paper ``Asymptotic properties of U-processes under
long-range dependence''}
\end{center}

This supplement contains proofs of Lemmas \ref{lem:O_P(1)}, \ref{lem:reste1},
    \ref{lem:controle_hermite}, \ref{lem:ext:soulier}, \ref{l:bar-gamma}, 
\ref{lem:cumulants}, \ref{lem:partie_qui_compte}, \ref{lem:reste} and 
\ref{lem:reste_D<1/2} and a section containing numerical experiments
illustrating some results of Section \ref{sec:appli}.

\begin{proof}[Proof of Lemma \ref{lem:O_P(1)}]
Let us check that the assumptions of Theorem  9 in \cite{Arcones:1994}
hold for the class $\mathcal{F}$ of functions $\{h_1(\cdot,r): r\in
I\}$ which is of rank $\tau\geq m=2> 1/D$.
By Assumption (A\ref{assum:long-range}) and since $\tau> 1/D$, the condition (i'') of Theorem 9 in \cite{Arcones:1994}
is satisfied. We conclude the proof of the Lemma by observing
that the condition (ii) of this Theorem is also fulfilled.
To check this condition, we have to prove that
$$
\int_0^{\infty} (N_{[~]}^{(2)}(\varepsilon,\mathcal{F}))^{1/2} \rmd\varepsilon <\infty\; ,
$$
where $N_{[~]}^{(2)}(\varepsilon,\mathcal{F})$ is the bracketing number of the class
$\mathcal{F}$ as defined on page 2269 in \cite{Arcones:1994}:
\begin{multline*}
N_{[~]}^{(2)}(\varepsilon,\mathcal{F})=\min\{N:\exists \textrm{
  measurable functions } f_1,\dots,f_N \textrm{ and }
\Delta_1,\dots,\Delta_N\\
\textrm{ such that } 
\textrm{ for each }f\in\mathcal{F},\exists i\leq
N\textrm{ such that } |f_i-f|\leq\Delta_i \textrm{ and where }\\ 
\PE(\Delta_i^2(X))\leq\varepsilon^2 \textrm{ for
  each } i\leq N \}\; .
\end{multline*}
Let $\{r_i, i=0,\dots,N\}$ be such that: for all $i$,
$|r_i-r_{i-1}|\leq \varepsilon/C$ and for all $r\in I$,
there exists $i$ such that $|r-r_i|\leq \varepsilon$. The smallest $N$
satisfying this property is at most equal to $[|I|C/\varepsilon]+1$, where $|I|$ denotes the length of $I$.
Let us define for all $i\geq 1$, $f_i=h_1(\cdot,r_i)$ and
$\Delta_i=h_1(\cdot,r_i)-h_1(\cdot,r_{i-1})$.
Using (\ref{assum:h_1_Lipschitz}) we first get
$\PE(\Delta_i^2(X))=\PE(\{h_1(X,r_i)-h_1(X,r_{i-1})\}^2)\leq
C^2|r_i-r_{i-1}|^{2}\leq\varepsilon^2$.
Now, let $r\in I$ and $1\leq i\leq N$ be such that,
$r_{i-1}\leq r\leq r_i$. Then, using the fact that $h_1$ is increasing with respect to its second argument leads to
$|f_i-h_1(\cdot,r)|\leq\Delta_i$. Thus,
$N_{[~]}^{(2)}(\varepsilon,\mathcal{F})\leq [|I|C/\varepsilon]+1$
which yields condition (ii) of Theorem 9 in \cite{Arcones:1994}.
\end{proof}

\begin{proof}[Proof of Lemma \ref{lem:reste1}]
We want to apply Lemma 5.2, P. 4307 of\\
\cite{borovkova:burton:dehling:2001} to $\{\sqrt{n} R_n(r),r\in
I\}$. To do this, we prove that for all
$s,t\in I,\delta>0$ such that $s\leq t\leq s+\delta$ and $s+\delta\in I$:
\begin{multline}\label{eq:decomp_Rn}
\sqrt{n}|R_n(t)-R_n(s)|\\
\leq\sqrt{n}|R_n(s+\delta)-R_n(s)|+2\sqrt{n}|W_n(s+\delta)-W_n(s)|
+4\sqrt{n}|U(s+\delta)-U(s)|\; .
\end{multline}
Using the definition of $R_n$ given by (\ref{eq:R_n_hoeff}) and the fact that
$h$, $h_1$ and $U$ are nondecreasing functions with respect to $r$, we get:
\begin{multline*}
R_n(t)-R_n(s)\leq\frac{1}{n(n-1)}\sum_{1\leq i\neq j\leq n}\{h(X_i,X_j,t)-h(X_i,X_j,s)\}
+\{U(t)-U(s)\}\\
\leq\frac{1}{n(n-1)}\sum_{1\leq i\neq j\leq n}\{h(X_i,X_j,s+\delta)-h(X_i,X_j,s)\}
+\{U(s+\delta)-U(s)\}\; .
\end{multline*}
By adding and subtracting functions $h_1$ evaluated at $s+\delta$ and $s$, we
obtain:
$$
R_n(t)-R_n(s)\leq \{R_n(s+\delta)-R_n(s)\}+\frac{2}{n}\sum_{i=1}^n\{h_1(X_i,s+\delta)-h_1(X_i,s)\}\; .
$$
Adding and subtracting $2(U(s)-U(s+\delta))$ leads to:
$$
R_n(t)-R_n(s)\leq \{R_n(s+\delta)-R_n(s)\}+\{W_n(s+\delta)-W_n(s)\}
+2\{U(s+\delta)-U(s)\}\; ,
$$
where $W_n$ is defined in (\ref{eq:W_n_hoeff}). Thus,
\begin{equation}\label{e:R1}
R_n(t)-R_n(s)\leq |R_n(s+\delta)-R_n(s)|+|W_n(s+\delta)-W_n(s)|
+2|U(s+\delta)-U(s)|\; .
\end{equation}
Let us now find an upper bound for $R_n(s)-R_n(t)$.
Starting with the expression (\ref{eq:R_n_hoeff}) for $R_n(r)$ and
setting $h(X_i,X_j,s)\leq h(X_i,X_j,s+\delta)$, $U(s)\leq U(s+\delta)$
and $h(X_i,X_j,t)\geq h(X_i,X_j,s)$, $U(t)\geq U(s)$ since $h$ and $U$
are non decreasing functions with respect to $r$, we obtain
\begin{multline*}
R_n(s)-R_n(t)\leq\frac{1}{n(n-1)}\sum_{1\leq i\neq j\leq n}
[\{h(X_i,X_j,s+\delta)-h(X_i,X_j,s)\}\\-2\{h_1(X_i,s)-h_1(X_i,t)\}]
+\{U(s+\delta)-U(s)\}\\
\leq\frac{1}{n(n-1)}\sum_{1\leq i\neq j\leq n}[
\{h(X_i,X_j,s+\delta)-h(X_i,X_j,s)\}\\
-2\{h_1(X_i,s+\delta)-h_1(X_i,s)\}
]\\
+\frac{2}{n(n-1)}\sum_{1\leq i \neq j\leq n}[
\{h_1(X_i,t)-h_1(X_i,s)\}
+\{h_1(X_i,s+\delta)-h_1(X_i,s)\}
]\\
+\{U(s+\delta)-U(s)\}\\
\leq \{R_n(s+\delta)-R_n(s)\}+\frac{4}{n}
\sum_{1\leq i\leq n}\{h_1(X_i,s+\delta)-h_1(X_i,s)\}\; .
\end{multline*}
Adding and subtracting $4(U(s)-U(s+\delta))$ leads to:
\begin{multline}\label{e:R2}
R_n(s)-R_n(t)\leq |R_n(s+\delta)-R_n(s)|+2|W_n(s+\delta)-W_n(s)|
+4|U(s+\delta)-U(s)|\; .
\end{multline}
Combining (\ref{e:R1}) and (\ref{e:R2}), we get for all $s,t\in
I,\delta>0$ such that $s\leq t\leq s+\delta$ and $s+\delta\in I$:
\begin{multline*}
\sqrt{n} |R_n(t)-R_n(s)|\leq \sqrt{n}|R_n(s+\delta)-R_n(s)|+2\sqrt{n}|W_n(s+\delta)-W_n(s)|\\
+4\sqrt{n}|U(s+\delta)-U(s)|\; ,
\end{multline*}
which is \eqref{eq:decomp_Rn}. Remark that $U$ is Lipschitz by
Condition (\ref{assum:h_1_Lipschitz}).
In Lemma 5.2, P. 4307 of
\cite{borovkova:burton:dehling:2001}, the monotone Lipschitz-continuous function
$\Lambda$ is here $U$, $\alpha=1/2$ and the process $\{Y_n(t)\}$ is here $\{\sqrt{n}W_n(t)\}$.
We shall now verify that conditions (i) and (ii) of that lemma are
satisfied. Condition (i) holds because of Lemma \ref{eq:maj_Rn}. Condition
(ii) involves $\{\sqrt{n}W_n(t)\}$.
%
Applying inequality (2.43) of Theorem 4 in \cite{Arcones:1994} to
$f(\cdot)=(h_1(\cdot,t)-h_1(\cdot,s))-(U(t)-U(s))$, which is, by
(\ref{eq:tau}), of Hermite rank $\tau\geq 2 > 1/D$, we get
using \eqref{assum:h_1_Lipschitz}
that there exist some positive constants $C$ and $C'$ such that:
\begin{multline*}
\PE\left[|\sqrt{n}\{W_n(t)-W_n(s)\}|^2\right]\\
=\PE\left[\left\{\frac{2}{\sqrt{n}}\sum_{i=1}^n
    (h_1(X_i,t)-h_1(X_i,s)) - (U(t)-U(s))\right\}^2\right]\\
\leq C \;\PE\left[\left\{\left(h_1(X_1,t)-h_1(X_1,s)\right)-\left(U(t)-U(s)\right)\right\}^2\right]
\leq C'\; |t-s|^2\; .
\end{multline*}
Thus condition (ii) of Lemma 5.2 in \cite{borovkova:burton:dehling:2001}
is satisfied with $r=2$ and monotone function $g(t)=t$. An application
of this lemma concludes the proof.
\end{proof}

\begin{proof}[Proof of Lemma \ref{lem:controle_hermite}]
Using that, for $n\geq 1$, $(H_{n-1}\varphi)'=-H_n\varphi$, where $'$
denotes the first derivative, and 6 integrations by parts (3 with
respect to $x$ and 3 with respect to $y$), we get that for  $p,q\geq 3$,
$$
c_{p,q}(f)=\int_{\rset^2}\partial^6 f(x,y)/\partial x^3\partial y^3
H_{q-3}(y)\varphi(y)H_{p-3}(x)\varphi(x)\rmd x\rmd y\; ,
$$
where $\varphi$ is the p.d.f of a standard Gaussian random variable.
(\ref{eq:c_pq:bound}) then follows from the Cauchy-Schwarz inequality.
\end{proof}
\begin{proof}[Proof of Lemma \ref{lem:ext:soulier}]
Let $f$ be a function defined on $\mathbb{R}^{a_1+a_2}$ such that
$f(Y)=f(Y_1,Y_2)=f_1(Y_1)f_2(Y_2)$.
Note that $\PE[f(Y)]=\PE[f(\Gamma^{1/2}_0 Z)]$, where the covariance
matrix of $Z$ is equal to $\Gamma^{-1/2}_0\Gamma\Gamma^{-1/2}_0$.
By the assumption on $r^{\star}$, the latter matrix is invertible
and satisfies: \\
$(\Gamma^{-1/2}_0\Gamma\Gamma^{-1/2}_0)^{-1}
=[I_{a_1+a_2}-\Gamma^{-1/2}_0(\Gamma_0-\Gamma)\Gamma^{-1/2}_0]^{-1}
=I_{a_1+a_2}+\sum_{k\geq 1} (\Gamma^{-1/2}_0(\Gamma_0-\Gamma)\Gamma^{-1/2}_0)^k$.
Let
$\Delta:=(\Gamma^{-1/2}_0\Gamma\Gamma^{-1/2}_0)^{-1}-I_{a_1+a_2}.$
By definition of the density of the
multivariate Gaussian distribution and the definition of the matrix $\Delta$, we obtain that
\begin{multline*}
|\Gamma^{-1/2}_0\Gamma\Gamma^{-1/2}_0|^{1/2}\;\PE[f(\Gamma^{1/2}_0
Z)]\\=\int_{\mathbb{R}^{a_1+a_2}} f(\Gamma^{1/2}_0 z) \exp(-z^T\Delta
z/2) \exp(-z^T z/2) \frac{\rmd z}{(2\pi)^{(a_1+a_2)/2}}\; .
\end{multline*}
Expanding $\exp(-z^T\Delta z/2)$ in series leads to
\begin{multline*}
|\Gamma^{-1/2}_0\Gamma\Gamma^{-1/2}_0|^{1/2}\;\PE[f(\Gamma^{1/2}_0
Z)]\\
=\sum_{k\geq 0}\frac{(-1/2)^k}{k!}\int_{\mathbb{R}^{a_1+a_2}} f(\Gamma^{1/2}_0 z) (z^T\Delta
z)^k \exp(-z^T z/2) \frac{\rmd z}{(2\pi)^{(a_1+a_2)/2}}\; .
\end{multline*}
Set $\nu=[(\tau+1)/2]$, where $[x]$ denotes
the integer part of $x$. Using that $f$ is of
Hermite rank at least $\tau$ and the previous equation, we get
\begin{multline*}
|\Gamma^{-1/2}_0\Gamma\Gamma^{-1/2}_0|^{1/2}\;\PE[f(\Gamma^{1/2}_0
Z)]\\
=\sum_{k\geq \nu}\frac{(-1/2)^k}{k!}\int_{\mathbb{R}^{a_1+a_2}} f(\Gamma^{1/2}_0 z) (z^T\Delta
z)^k \exp(-z^T z/2) \frac{\rmd z}{(2\pi)^{(a_1+a_2)/2}}\; .
\end{multline*}
Since $|\sum_{k\geq\nu}(-1/2)^k (z^T\Delta z)^k/k!|
\leq |z^T\Delta z|^{\nu}\exp(|z^T\Delta z|/2)/(2^{\nu}\nu!)$, we
obtain
\begin{multline*}
|\Gamma^{-1/2}_0\Gamma\Gamma^{-1/2}_0|^{1/2}\;|\PE[f(\Gamma^{1/2}_0
Z)]|\\
\leq \frac{1}{2^{\nu} \nu!}\int_{\mathbb{R}^{a_1+a_2}}
|f(\Gamma^{1/2}_0 z)|\;|z^T\Delta z|^{\nu}\exp(|z^T\Delta z|/2)
\exp(-z^T z/2)\frac{\rmd z}{(2\pi)^{(a_1+a_2)/2}}\; .
\end{multline*}
Denoting by $\delta$ the spectral radius of $\Delta$ gives
\begin{multline*}
|\Gamma^{-1/2}_0\Gamma\Gamma^{-1/2}_0|^{1/2}\;|\PE[f(\Gamma^{1/2}_0
Z)]|\\
\leq \frac{\delta^\nu}{2^{\nu} \nu!}\int_{\mathbb{R}^{a_1+a_2}}
|f(\Gamma^{1/2}_0 z)| (z^T z)^{\nu} \exp\{(\delta/2-1/2)z^T z\}
\frac{\rmd z}{(2\pi)^{(a_1+a_2)/2}}\; .
\end{multline*}
By the Cauchy-Schwarz inequality, we get
\begin{multline}\label{eq:bound_CS}
|\Gamma^{-1/2}_0\Gamma\Gamma^{-1/2}_0|^{1/2}\;|\PE[f(\Gamma^{1/2}_0
Z)]|\\
\leq \frac{\delta^\nu}{2^{\nu} \nu!}\left(\int_{\mathbb{R}^{a_1+a_2}}
f^2(\Gamma^{1/2}_0 z)\exp(-z^T z/2)\frac{\rmd z}{(2\pi)^{(a_1+a_2)/2}}\right)^{1/2}\\
\left(\int_{\mathbb{R}^{a_1+a_2}}(z^T z)^{2\nu}\exp\{(\delta-1/2)z^T z\}
\frac{\rmd z}{(2\pi)^{(a_1+a_2)/2}}\right)^{1/2}\; .
\end{multline}
By definition of $\Delta$, the spectral radius $\delta$ of $\Delta$
satisfies $\delta\leq \sum_{k\geq 1}(r^{\star})^k\leq
r^{\star}/(1-r^{\star})$, where $r^{\star}$ is the
spectral radius of $\Gamma^{-1/2}_0(\Gamma_0-\Gamma)\Gamma^{-1/2}_0$.
By assumption on $r^{\star}$, $\delta\leq 1/2-3\varepsilon/2$ which
implies that the second integral in (\ref{eq:bound_CS}) is convergent.
The first integral in (\ref{eq:bound_CS}) satisfies
\begin{equation}\label{e:f}
\left(\int_{\mathbb{R}^{a_1+a_2}}
f^2(\Gamma^{1/2}_0 z)\exp(-z^T z/2)\frac{\rmd  z}{(2\pi)^{(a_1+a_2)/2}}\right)^{1/2}
=\|f_1\|_{2,\Gamma_{0,1}} \|f_2\|_{2,\Gamma_{0,2}}\; .
\end{equation}
Finally, under the assumption on $r^{\star}$, the spectral radius
$\delta_0$ of \\
$(\Gamma^{-1/2}_0\Gamma\Gamma^{-1/2}_0)^{-1}=\sum_{k\geq
  0} \{\Gamma^{-1/2}_0(\Gamma_0-\Gamma)\Gamma^{-1/2}_0\}^k$
satisfies $\delta_0\leq\sum_{k\geq 0} (r^{\star})^k=1/(1-r^{\star})$,
so that
$|\Gamma^{-1/2}_0\Gamma\Gamma^{-1/2}_0|^{-1/2}\leq (1-r^{\star})^{-(a_1+a_2)/2}
\leq (2/3+\varepsilon)^{-(a_1+a_2)/2}\leq (3/2)^{(a_1+a_2)/2}$.
This establishes (\ref{eq:soulier_result}).
\end{proof}
\begin{proof}[Proof of Lemma \ref{l:bar-gamma}]
Let $\textrm{I}$ denote the $2\times 2$ identity matrix. One can
express $\Gamma_{11}$ as $\Gamma_{11}=L_aL_a^T$, where $T$ denotes
the transpose so that the vector
$(\bar{X}_1,\bar{X}_2)^T=L_a^{-1}(X_1,X_2)^T$ has covariance matrix
$\bar{\Gamma}_{11}=L_a^{-1}\Gamma_{11}(L_a^{-1})^T=\textrm{I}$.
Similarly, $\Gamma_{22}=L_b L_b^T$ so that 
$(\bar{X}_3,\bar{X}_4)^T=L_b^{-1} (X_3,X_4)^T$ has covariance matrix
$\bar{\Gamma}_{22}=\textrm{I}$. Then
$$
\PE_{\Gamma}[J_a(X_1,X_2) J_b(X_3,X_4)]
=\PE_{\bar{\Gamma}}[\bar{J}_a(\bar{X}_1,\bar{X}_2)
\bar{J}_b(\bar{X}_3,\bar{X}_4)]\; ,
$$
where $\bar{J}_a=J_a\circ L_a$, $\bar{J}_b=J_b\circ L_b$ and
\begin{multline*}
\bar{\Gamma}=
\begin{bmatrix}
L_a^{-1}&0\\
0 & L_b^{-1}\\
\end{bmatrix}
\begin{bmatrix}
\Gamma_{11}&\Gamma_{12}\\
\Gamma_{21}&\Gamma_{22}\\
\end{bmatrix}
\begin{bmatrix}
(L_a^{-1})^T&0\\
0 & (L_b^{-1})^T\\
\end{bmatrix}\\
=\begin{bmatrix}
\textrm{I}&L_a^{-1}\Gamma_{12}(L_b^{-1})^T\\
L_b^{-1}\Gamma_{21}(L_a^{-1})^T&\textrm{I}\\
\end{bmatrix}\; .
\end{multline*}
Observe that
$$
L_a=
\begin{bmatrix}
1&0\\
\rho_{12}&\sqrt{1-\rho_{12}^2}\\
\end{bmatrix}
\; ,\;
L_a^{-1}=
\begin{bmatrix}
1&0\\
-\rho_{12}/\sqrt{1-\rho_{12}^2}&1/\sqrt{1-\rho_{12}^2}\\
\end{bmatrix}
$$
since $L_aL_a^T=\Gamma_{11}$. A similar expression holds for $L_b$
with $\rho_{12}$ replaced by $\rho_{34}.$
Observe that
\begin{multline*}
\bar{\Gamma}_{12}=L_a^{-1}\Gamma_{12}(L_b^{-1})^T
=\\
\begin{bmatrix}
1&0\\
-\rho_{12}/\sqrt{1-\rho_{12}^2}&1/\sqrt{1-\rho_{12}^2}\\
\end{bmatrix}
\begin{bmatrix}
\rho_{13}&\rho_{14}\\
\rho_{23}&\rho_{24}\\
\end{bmatrix}
\begin{bmatrix}
1&-\rho_{34}/\sqrt{1-\rho_{34}^2}\\
0&1/\sqrt{1-\rho_{34}^2}\\
\end{bmatrix}
\\
=
\begin{bmatrix}
\bar{\rho}_{13}&\bar{\rho}_{14}\\
\bar{\rho}_{23}&\bar{\rho}_{24}\\
\end{bmatrix}
\; ,
\end{multline*}
where the $\bar{\rho}_{ij}$ are given in the statement of the lemma.
This characterizes the matrix
$\bar{\Gamma}$ since $\bar{\Gamma}_{11}=\bar{\Gamma}_{22}=\textrm{I}$
and $\bar{\Gamma}_{21}=\bar{\Gamma}_{12}^T$.
The relations involving $\rho^\star$ and $\bar{\rho}^\star$ follow
from those relating the $\bar{\rho}_{ij}$'s to the
$\rho_{kl}$'s. Finally, one has 
$$
\PE_{\textrm{I}}[\bar{J}_a(\bar{X}_1,\bar{X}_2)^2]
=\PE_{\Gamma_{11}}[J_a(X_1,X_2)^2]\; ,
$$
and also the other similar type relations.
\end{proof}
\begin{proof}[Proof of Lemma \ref{lem:cumulants}]
We first prove that, for $p\geq 2$, the $p$th cumulant $\kappa_p$ of $a
Z_{2,D}(1)+b (Z_{1,D}(1))^2$ is equal to
\begin{multline}\label{eq:cumulants}
\kappa_p=2^{p-1}(p-1)!\; k(D)^p\int_{[0,1]^p} \rmd u_1\dots \rmd u_p \int_{[0,1]^p}
\rmd v_1\dots \rmd v_p\\
\prod_{j=1}^p [a\delta(u_j-v_j)+b] \; |u_j-v_{j-1}|^{-D}, \textrm{ with }
v_0=v_p \; ,
\end{multline}
where $k(D)=\B((1-D)/2,D)$, $\delta(x)=1$ if $x=0$, and $\delta(x)=0$ else.
Using (\ref{eq:fBm}) and (\ref{eq:rosenblatt}),
$$
a Z_{2,D}(1)+b (Z_{1,D}(1))^2=\int'_{\rset^2} K(x,y) \rmd B(x)
\rmd B(y)+b\sigma^2\; ,
$$
where
$$
K(x,y)=\int_0^1\int_0^1 [a\delta(u-v)+b] (u-x)_+^{-(D+1)/2}
(v-y)_+^{-(D+1)/2} \rmd u \rmd v\; ,
$$
and
\begin{multline*}
\sigma^2=E[Z_{1,D}(1)^2]=\int_\rset[\int_0^1
(u-x)_+^{-\frac{D+1}{2}} \rmd u]^2 \rmd x\\
=\int_0^1\int_0^1[\int_\rset
(u-x)_+^{-\frac{D+1}{2}}(v-x)_+^{-\frac{D+1}{2}}\rmd x] \rmd u \rmd v\; .
\end{multline*}
Using that for $0<\alpha<1/2$,
\begin{multline}\label{eq:beta}
\int_{\rset} (u-x)_+^{\alpha-1} (v-x)_+^{\alpha-1} \rmd x
=|u-v|^{2\alpha-1}\int_0^{\infty}y^{\alpha-1}(1+y)^{\alpha-1}
\rmd y\\=|u-v|^{2\alpha-1}\B(\alpha,-2\alpha+1) \; ,
\end{multline}
where $\B(\cdot,\cdot)$ denotes the Beta function, we get
\begin{multline}\label{e:bet}
\int_{\rset} (u-x)_+^{-(D+1)/2} (v-x)_+^{-(D+1)/2} \rmd x \\=
\B((1-D)/2,D) |u-v|^{-D} = k(D) |u-v|^{-D} \; .
\end{multline}
Thus,
\begin{equation}\label{e:VarFBM}
\sigma^2=k(D)\int_0^1\int_0^1 |u-v|^{-D} \rmd u
\rmd v=\frac{2k(D)}{(-D+1)(-D+2)}\; .
\end{equation}
Hence, using Proposition 4.2 in \cite{fox:taqqu:1987}, we have that for $p\geq 2$,
$$
\kappa_p=2^{p-1}(p-1)!\int_{\rset^p} K(x_1,x_2) K(x_2,x_3)\dots
K(x_{p-1},x_p) K(x_p,x_1) \rmd x_1\dots \rmd x_p\; .
$$
By definition of $K$, and with the convention $x_{p+1}=x_1$,
\begin{multline*}
\int_{\rset^p} K(x_1,x_2) K(x_2,x_3)\dots
K(x_{p-1},x_p) K(x_p,x_1) \rmd x_1\dots \rmd x_p\\
=\int_{[0,1]^p} \rmd u_1\dots \rmd u_p \int_{[0,1]^p} \rmd v_1\dots \rmd v_p
\int_{\rset^p} \rmd x_1\dots \rmd x_p \\
\prod_{j=1}^p [a\delta(u_j-v_j)+b] (u_j-x_j)_+^{-(D+1)/2}
(v_j-x_{j+1})^{-(D+1)/2}\\
= \int_{[0,1]^p} \rmd u_1\dots \rmd u_p \int_{[0,1]^p} \rmd v_1\dots \rmd v_p\\
\prod_{j=1}^p [a\delta(u_j-v_j)+b] \int_{\rset} (u_j-x_j)_+^{-(D+1)/2}
(v_{j-1}-x_{j})^{-(D+1)/2}
\end{multline*}
where $v_0=v_p$ since $x_j$ is associated with $u_j$ and $v_{j-1}$.
Using (\ref{e:bet}), we obtain
the expression (\ref{eq:cumulants}) for the cumulants $\kappa_p$, $p\geq 2$.
Let us now compute the limit as $n$ tends to infinity of the cumulants of
$$
A_n=\frac{n^{D-2}}{L(n)}\left[X'(a n \textrm{I}+b \mathbf{1}\mathbf{1}')X\right]
=\frac{n^{D-2}}{L(n)}\left[\widetilde{X}'
\Sigma^{1/2}\left(a n \textrm{I}+b \mathbf{1}\mathbf{1}'\right)\Sigma^{1/2}\widetilde{X}\right]\; ,
$$
where $X=(X_1,\dots,X_n)'$, $\mathbf{1}=(1,\dots,1)'$, $\textrm{I}$ is the $n\times n$
identity matrix, $\Sigma$ is the covariance matrix
of $X$ and $\widetilde{X}$ is a standard Gaussian random vector. Using \cite{stuart:ord:1987}, p. 488, the $p$th cumulant of
$A_n$ is equal to
$$
cum_p=2^{p-1} (p-1)!\; \trace(B_n^p)\; ,
$$
where $B_n=n^{D-2}L(n)^{-1}\left[\Sigma^{1/2}\left(a n \textrm{I}+b \mathbf{1}\mathbf{1}'\right)\Sigma^{1/2}\right].$
But
\begin{multline*}
\trace(B_n^p)=\left(\frac{n^{D-2}}{L(n)}\right)^p
\trace\left[\left\{(a n \textrm{I}+b \mathbf{1}\mathbf{1}')\Sigma\right\}^p\right]\\
=\left(\frac{n^{D-2}}{L(n)}\right)^p
\sum_{\stackrel{1\leq i_1,i_2,\dots,i_p\leq n}{1\leq j_1,j_2,\dots,j_p\leq n}}
D_{i_1,j_1}\rho(j_1-i_2) D_{i_2,j_2} \rho(j_2-i_3)\dots \\
D_{i_{p-1},j_{p-1}} \rho(j_{p-1}-i_p)D_{i_p,j_p} \rho(j_p-i_1)\; ,
\end{multline*}
where $\rho$ is defined in Assumption (A\ref{assum:long-range}) and $D_{i,j}=an\delta(i-j)+b$.
With the convention $i_{p+1}=i_1$,
\begin{multline*}
\trace(B_n^p)=\frac{1}{n^{2p}}\sum_{\stackrel{1\leq i_1,i_2,\dots,i_p\leq n}{1\leq j_1,j_2,\dots,j_p\leq n}}
\prod_{\ell=1}^p \left\{\frac{n^D}{L(n)}[an\delta(i_\ell-j_\ell)+b]\rho(j_\ell-i_{\ell+1})\right\}\\
=\frac{1}{n^{2p}}\sum_{\stackrel{1\leq i_1,i_2,\dots,i_p\leq n}{1\leq j_1,j_2,\dots,j_p\leq n}}
\prod_{\ell=1}^p \left\{\frac{n^D}{L(n)}[an\delta(i_\ell-j_\ell)+b]\rho(j_{\ell-1}-i_\ell)\right\}\; ,
\end{multline*}
where $j_0=j_p$. Thus, as $n$ tends to infinity,
\begin{multline*}
\trace(B_n^p)\to \int_{[0,1]^p} \rmd u_1\dots \rmd u_p \int_{[0,1]^p} \rmd v_1\dots \rmd v_p\\
\prod_{j=1}^p [a\delta(u_j-v_j)+b] \int_{\rset} |u_j-v_{j-1}|^{-D}\; ,
\end{multline*}
with the convention $v_0=v_p$, which gives the expected result.
\end{proof}
\begin{proof}[Proof of Lemma \ref{lem:partie_qui_compte}]
By Assumption (A\ref{assum:long-range}), $\rho(k)=k^{-D} L(k)$.
Using the adaptation of Karamata's theorem given in \cite{taqqu:1975},
one gets
\begin{equation}\label{eq:taqqu_trick}
\textrm{if } D<1/m,\; \sum_{|k|<n}|\rho(k)|^{m}\sim \frac{2}{(1-mD)(2-mD)}n^{1-mD}(L(n))^{m}\; .
\end{equation}
The bound (\ref{case_(i)}) follows from
$$
\PE[\{\sum_{i=1}^n X_i\}^2]\leq  (n+\sum_{1\leq i\neq j\leq
  n}|\rho(i-j)|)
\leq n(1+\sum_{|k|<n, k\neq 0} |\rho(k)|)
\; ,
$$
and by setting $m=1$ in (\ref{eq:taqqu_trick}).
Let us prove (\ref{case_(ii)_a}). Given that
$\PE[H_p(X_i)H_q(X_j)]=p!\; \delta(p-q) \rho(i-j)^p$, for all integers
$p,q,i,j\geq 1$, we obtain
\begin{multline*}
\PE[\{\sum_{i=1}^n (X_i^2-1)\}^2]=\PE[\sum_{1\leq i,j\leq n} H_2(X_i) H_2(X_j)]\\
= 2 n + 2\sum_{1\leq i\neq j\leq n} \rho(i-j)^2
\leq 2n(1+\sum_{|k|<n, k\neq 0} \rho(k)^2)
\; .
\end{multline*}
The bound (\ref{case_(ii)_a}) follows by using that $D<1/2$ and
(\ref{eq:taqqu_trick}) with $m=2$.
Let us now prove (\ref{case_(ii)_b}). Note that
\begin{multline}\label{eq:compte_2}
\PE[(\sum_{1\leq i\neq j\leq n}X_i X_j)^2]
=\sum_{\stackrel{1\leq i\neq j\leq n}{1\leq k\neq \ell\leq n}}\PE(X_i X_j X_k X_{\ell})\\
=\sum_{1\leq i\neq j\leq n} \PE(X_i^2 X_j^2)
+\sum_{\stackrel{1\leq i,j,k,\ell\leq n}{|\{i,j,k,\ell\}|=4}} \PE(X_i X_j X_k X_{\ell})
+6\sum_{\stackrel{1\leq i,j,\ell\leq n}{|\{i,j,\ell\}|=3}} \PE(X_i^2 X_j X_{\ell})\; .
\end{multline}
Writing $X_i^2=H_2(X_i)+1$ and using that $\PE[H_2(X_i)H_2(X_j)]=2 \rho(i-j)^2$, for all $i,j\geq 1$,
the first term in the r.h.s of \eqref{eq:compte_2} satisfies
$$
\sum_{1\leq i\neq j\leq n} \PE(X_i^2 X_j^2)\leq n^2+2n\sum_{|k|<n,k\neq 0} \rho(k)^2\; .
$$
Using Lemma 3.2 P. 210 in \cite{taqqu:1977}, the second term
in the r.h.s of \eqref{eq:compte_2} satisfies, for some positive
constant $C$,
$$
\sum_{\stackrel{1\leq i,j,k,\ell\leq n}{|\{i,j,k,\ell\}|=4}} \PE(X_i
X_j X_k X_{\ell})
\leq C n^2 (\sum_{|k|<n,k\neq 0} |\rho(k)|)^2\; .
$$
Writing $X_i^2=H_2(X_i)+1$ and using Lemma 3.2 P. 210 in \cite{taqqu:1977}, the third term
in the r.h.s of \eqref{eq:compte_2} satisfies, for some positive
constant $C$,
$$
\sum_{\stackrel{1\leq i,j,\ell\leq n}{|\{i,j,\ell\}|=3}} \PE(X_i^2 X_j X_{\ell})
\leq C n(\sum_{|k|<n} |\rho(k)|)^2+n^2\sum_{|k|<n} |\rho(k)|\; .
$$
The last three inequalities lead to the expected result by using (\ref{eq:taqqu_trick}).
\end{proof}
\begin{proof}[Proof of Lemma \ref{lem:reste}]
Set $\alpha_{p,q}(s,t)=\alpha_{p,q}(t)-\alpha_{p,q}(s)$ for all
$s,t$ in $\rset$, where $\alpha_{p,q}(\cdot)$ is defined in \eqref{e:alpha}.
Then
\begin{equation}\label{eq:second_decomp}
\PE[(\widetilde{R}_n(t)-\widetilde{R}_n(s))^2]=
\sum_{\stackrel{1\leq i_1\neq i_2\leq n}{1\leq i_3\neq i_4\leq n}}
\PE[\widetilde{J}_{s,t}(X_{i_1},X_{i_2})\widetilde{J}_{s,t}(X_{i_3},X_{i_4})]\; ,
\end{equation}
where for all $x,y$ in $\rset$ and $s,t$ in $I$,
\begin{multline}\label{e:Jm=1}
\widetilde{J}_{s,t}(x,y)=(h(x,y,t)-h(x,y,s))-(\alpha_{1,0}(t)-\alpha_{1,0}(s))(x+y)\\-(U(t)-U(s))\;
, \textrm{ if } m=1\; , 
\end{multline}
\begin{multline}\label{e:Jm=2}
\widetilde{J}_{s,t}(x,y)=(h(x,y,t)-h(x,y,s))-(\alpha_{1,1}(t)-\alpha_{1,1}(s))xy\\
-\frac12(\alpha_{2,0}(t)-\alpha_{2,0}(s))(x^2+y^2-2)
-(U(t)-U(s))\;
, \textrm{ if } m=2\; .
\end{multline}
To obtain these relations express $\widetilde{R}_n(t)-\widetilde{R}_n(s)$
using (\ref{eq:herm_decomp_Un}), (\ref{e:Uh}) and (\ref{def:W_n}).
We now consider 3 cases, depending on the cardinality of the set
$\{i_1,i_2,i_3,i_4\}$.

1) We start with the case of cardinality 2.
Let us address the case where the sum is over the set of indices
$\{i_1,i_2,i_3,i_4\}$ such that $i_1=i_3$ and
$i_2=i_4$. We shall
only focus on the case where $m=1$ because the case $m=2$ could be
addressed in the same way.
We thus need to show that
\begin{equation}\label{e:Balpha}
\PE[(\widetilde{R}_{n}(t)-\widetilde{R}_{n}(s))^2]\leq C
n^{4-D-\alpha}|t-s|\; .
\end{equation}
Using that $h$, $U$ and $\alpha_{1,0}$
are bounded functions, there exists a positive constant $C$ such that
\begin{multline*}
\sum_{1\leq i_1\neq i_2\leq n}
\PE[\widetilde{J}_{s,t}^2(X_{i_1},X_{i_2})]\leq C
\sum_{1\leq i_1\neq i_2\leq n}\PE\left[|h(X_{i_1},X_{i_2},t)-h(X_{i_1},X_{i_2},s)|\right. \\
\left.+|\alpha_{1,0}(t)-\alpha_{1,0}(s)|\{(X_{i_1}+X_{i_2})^2+|X_{i_1}+X_{i_2}|\}
+|U(t)-U(s)|\right]\; .
\end{multline*}
Since Condition (\ref{eq:lip_h}) holds and $U$, $\widetilde{\Lambda}$,
defined in (\ref{eq:lambda_tilde})
are Lipschitz functions, there exist positive constants $C_1$ and
$C_2$ such that
\begin{multline*}
\sum_{1\leq i_1\neq i_2\leq n}
\PE[\widetilde{J}_{s,t}^2(X_{i_1},X_{i_2})]\\
\leq C_1 n(n-1) |t-s|+C_2|t-s|\sum_{1\leq i_1\neq i_2\leq
  n}\PE[(X_{i_1}+X_{i_2})^2+|X_{i_2}|+|X_{i_2}|]\\
\leq C n(n-1) |t-s|\; ,
\end{multline*}
which gives (\ref{e:Balpha}).

2) Let us now address the case where the sum is over the set of indices
$\{i_1,i_2,i_3,i_4\}$ having a cardinal number equal to 3
\textit{i.e.} for instance when $i_3=i_1$. As previously, we
focus on the case where $m=1$.  Using that $h$, $U$ and $\alpha_{1,0}$
are bounded functions, Condition (\ref{eq:lip_h}), the Lipschitz
property of  $U$ and $\widetilde{\Lambda}$,
there exists a positive constant $C$ such that
$$
\sum_{\stackrel{1\leq i_1\neq i_2\leq n}{1\leq i_1\neq i_4\leq n}}
\PE[\widetilde{J}_{s,t}(X_{i_1},X_{i_2})\widetilde{J}_{s,t}(X_{i_1},X_{i_4})]
\leq C n^3 |t-s|\; ,
$$
which gives (\ref{e:Balpha}).

3) Let us now consider the case where the sum is over the indices
$i_1,i_2,i_3,i_4$ such that the cardinal number of the set
$\{i_1,i_2,i_3,i_4\}$ is equal to 4. 
This case is similar to the case 3) in the proof of Lemma
\ref{eq:maj_Rn}. We need to show that
\begin{equation}\label{e:goal}
a_n^2\sum_{\stackrel{1\leq i_1,i_2,i_3,i_4\leq n}{|\{i_1,i_2,i_3,i_4\}|=4}}
\PE[\widetilde{J}(X_{i_1},X_{i_2})\widetilde{J}(X_{i_3},X_{i_4})]
\leq C \frac{|t-s|}{n^{\alpha}}\; ,
\end{equation}
where $a_n=n^{mD/2-2}L(n)^{-m/2}$, $\widetilde{J}=\widetilde{J}_{s,t}$
is defined in (\ref{e:Jm=1}) if $m=1$ and in (\ref{e:Jm=2}) if $m=2$,
$C>0$ and $\alpha>0$. We will only present the case $m=1$. The idea
is, once again, to replace
$(X_{i_1},X_{i_2},X_{i_3},X_{i_4})$ by
$(\bar{X}_{i_1},\bar{X}_{i_2},\bar{X}_{i_3},\bar{X}_{i_4})$
using Lemma \ref{l:bar-gamma}, so that (\ref{e:JJbar}) holds with $J$,
$J_{i_1,i_2}$, $J_{i_3,i_4}$ replaced by $\widetilde{J}$,
$\widetilde{J}_{i_1,i_2}$, $\widetilde{J}_{i_3,i_4}$ and to expand
$\widetilde{J}_{i_1,i_2}(\bar{X}_{i_1},\bar{X}_{i_2})$ and
$\widetilde{J}_{i_3,i_4}(\bar{X}_{i_3},\bar{X}_{i_4})$ in Hermite
polynomials as in (\ref{e:JKsum}), with an expansion up to $K$, thus
defining $\widetilde{J}_{i_1,i_2}^K$ and  $\widetilde{J}_{i_3,i_4}^K$.
Then, (\ref{e:JJK}) holds with $J$ replaced by $\widetilde{J}$.
Denoting again the coefficients of the expansion by $c_{p_1,p_2}^{i_1,i_2}$
and $c_{p_3,p_4}^{i_3,i_4}$ respectively, one needs to majorize the
right-hand side of (\ref{e:bd1}).
Assuming that $|i_1-i_2|$ is the smallest distance between two
different indices, namely  $|i_1-i_2|=\min\{|i_1-i_2|,|i_1-i_3|,
|i_1-i_4|,|i_2-i_3|,|i_2-i_4|,|i_3-i_4|\}$, we obtain that the 
$\bar{\rho}_{ij}$'s are bounded by $4\rho(i_1-i_2)/(1-\rho(i_1-i_2)^2)$.
Using that there exist positive constants $C$ and $\varepsilon$ such that
$|\rho(k)|\leq C(1\wedge |k|^{-D+\varepsilon})=:\gamma(k)$,
for all $ k\geq 1$,
we get, by (\ref{e:rhobar}), that the $\bar{\rho}_{ij}$'s are bounded by 
\begin{equation}\label{e:gammabar}
4\gamma(i_1-i_2)/(1-\gamma(i_1-i_2)^2)=:\bar{\gamma}(i_1-i_2)\;.
\end{equation}
By Lemma 3.2 in \cite{taqqu:1977}, we obtain
\begin{multline*}
\PE_{\bar{\Gamma}}[H_{p_1}(\bar{X}_{i_1})H_{p_2}(\bar{X}_{i_2}) H_{p_3}(\bar{X}_{i_3}) H_{p_4}(\bar{X}_{i_4})]\\
\leq C \bar{\gamma}(i_1-i_2)^{\frac{p_1+p_2+p_3+p_4}{2}} |\PE[H_{p_1}(X) H_{p_2}(X) H_{p_3}(X) H_{p_4}(X)]|\; .
\end{multline*}
Using (\ref{eq:taqqu_bound}), $\PE_{\bar{\Gamma}}[\widetilde{J}_{i_1,i_2}^K(\bar{X}_{i_1},\bar{X}_{i_2})
\widetilde{J}_{i_3,i_4}^K(\bar{X}_{i_3},\bar{X}_{i_4})]$ is bounded by
$$
\sum_{1\leq p_1,p_2\leq K}
\frac{|c_{p_1,p_2}^{i_1,i_2}|}{\sqrt{p_1!\;p_2!}}(3\bar{\gamma}(i_1-i_2))^{\frac{p_1+p_2}{2}}
\sum_{1\leq p_3,p_4\leq K}
\frac{|c_{p_3,p_4}^{i_3,i_4}|}{\sqrt{p_3!\;p_4!}}(3\bar{\gamma}(i_1-i_2))^{\frac{p_3+p_4}{2}}\;.
$$
Using the Cauchy-Schwarz inequality, the first sum is bounded by
$$
\PE_{\textrm{I}}[\widetilde{J}_{i_1,i_2}(\bar{X}_{i_1},\bar{X}_{i_2})^2]^{1/2}
\left(\sum_{p\geq 1} (3\bar{\gamma}(i_1-i_2))^p\right)\;,
$$
and similarly for the second sum. It follows from Lemma
\ref{l:bar-gamma} that \\$\PE_{\bar{\Gamma}}[\widetilde{J}_{i_1,i_2}^K(\bar{X}_{i_1},\bar{X}_{i_2})
\widetilde{J}_{i_3,i_4}^K(\bar{X}_{i_3},\bar{X}_{i_4})]$ is bounded by
$$
\PE_{\Gamma_{11}}[\widetilde{J}(X_{i_1},X_{i_2})^2]^{1/2}
\PE_{\Gamma_{22}}[\widetilde{J}(X_{i_3},X_{i_4})^2]^{1/2}
\left(\sum_{p\geq 1} (3\bar{\gamma}(i_1-i_2))^p\right)^2\; .
$$
Since Condition (\ref{eq:lip_h}) holds and $U$, $\widetilde{\Lambda}$,
defined in (\ref{eq:lambda_tilde})
are Lipschitz functions,
$$
\PE_{\Gamma_{11}}[\widetilde{J}(X_{i_1},X_{i_2})^2]^{1/2}
\PE_{\Gamma_{22}}[\widetilde{J}(X_{i_3},X_{i_4})^2]^{1/2}
\leq C |t-s|\; .
$$
We deduce from the previous inequalities that
\begin{multline*}
a_n^2\sum_{\stackrel{1\leq i_1,i_2,i_3,i_4\leq n}{|\{i_1,i_2,i_3,i_4\}|=4}}
\PE[\widetilde{J}(X_{i_1},X_{i_2})\widetilde{J}(X_{i_3},X_{i_4})]\\
\leq C n^3 a_n^2 |t-s| \sum_{\stackrel{|k|<n}{k\neq 0}}\left(\sum_{p\geq 1}
  (3\bar{\gamma}(k))^p\right)^2\; ,
\end{multline*}
where $\bar{\gamma}$ is defined in \eqref{e:gammabar}.
Let $\eta$ be a positive
constant such that $\eta>3$.
Then there is $K\geq 1$ such that $\eta\bar{\gamma}(k)\leq 1$,
for all $k\geq K$. We may suppose without loss of generality
that $K=1$, that is $\eta\bar{\gamma}(k)\leq 1$, for all $k\geq 1$.
We then obtain that for large enough $n$,
\begin{equation}\label{eq:tilde_sum}
a_n^2\PE[(\widetilde{R}_n(t)-\widetilde{R}_n(s))^2]
\leq C\eta^{2} |t-s|\; a_n^2\; n^3
\left(\sum_{\stackrel{|k|<n}{k\neq 0}}\bar{\gamma}(k)^{2} \right)
\left(\sum_{p\geq 1}\left(\frac{3}{\eta}\right)^{p}\right)\; .
\end{equation}
Observe that $a_n^2 n^3=n^{D-1}L(n)^{-1}$.
Recall also that  $\bar{\gamma}(k)$ is defined in (\ref{e:gammabar}) and note
that $\sum_{k} \bar{\gamma}(k)^{2}$ may be finite or infinite.
If $\sum_{k} \bar{\gamma}(k)^{2}<\infty$ then the result (\ref{e:Rtilde}) follows with
$\alpha=1-D$ which is positive since $D<1$.
If  $\sum_{k} \bar{\gamma}(k)^{2}=\infty$ then $\sum_{|k|<n}
\bar{\gamma}(k)^{2}\sim 16 n^{1-2D+2\varepsilon}$ and the
result follows with $\alpha=D-2\varepsilon$ if $\varepsilon$ is
chosen small enough to ensure that this quantity is positive.
\end{proof}
\begin{proof}[Proof of Lemma \ref{lem:reste_D<1/2}]
We want to apply Lemma 5.2, P. 4307 of\\
\cite{borovkova:burton:dehling:2001} to $\{a_n \widetilde{R}_n(r),r\in
I\}$. To do so, we first prove that for all
$s,t\in I,\delta>0$ such that $s\leq t\leq s+\delta$ and $s+\delta\in I$:
\begin{multline}\label{eq:decomp_Rn_tilde_1}
a_n|\widetilde{R}_n(t)-\widetilde{R}_n(s)|
\leq a_n|\widetilde{R}_n(s+\delta)-\widetilde{R}_n(s)|
+2a_n |\widetilde{\Lambda}(s+\delta)-\widetilde{\Lambda}(s)|
|\sum_{1\leq i\neq j\leq n} X_i|\\
+2a_n n(n-1)|U(s+\delta)-U(s)|\; , \textrm{ if $m=1$ and $D<1$}\;,
\end{multline}
where $\widetilde{\Lambda}$ is defined in (\ref{eq:lambda_tilde}) and
\begin{multline}\label{eq:decomp_Rn_tilde_2}
a_n|\widetilde{R}_n(t)-\widetilde{R}_n(s)|
\leq a_n|\widetilde{R}_n(s+\delta)-\widetilde{R}_n(s)|
+2a_n|\widetilde{\Lambda}(s+\delta)-\widetilde{\Lambda}(s)|\\
[\,|\sum_{1\leq i\neq j\leq n} X_i X_j|
+|\sum_{1\leq i\neq j\leq n}(X_i^2-1)|\;]
+2a_n n(n-1)|U(s+\delta)-U(s)|\; , \\\textrm{ if $m=2$ and $D<1/2$}\; .
\end{multline}
Let us focus on the proof of \eqref{eq:decomp_Rn_tilde_1}, where $m=1$
and $D<1$, since the proof
of \eqref{eq:decomp_Rn_tilde_2} can be obtained by using similar arguments.
In view of the definition (\ref{eq:herm_decomp_Un}) of $\widetilde{R}_n$ and the fact that $U_n$ and $U$ are non
decreasing functions, we obtain
$$
\widetilde{R}_n(t)-\widetilde{R}_n(s)\leq \widetilde{R}_n(s+\delta)-\widetilde{R}_n(s)
+\widetilde{W}_n(s+\delta)-\widetilde{W}_n(t)+ n(n-1)(U(s+\delta)-U(s))\; .
$$
Remark that the monotonicity of $h$ in (\ref{eq:lambda_tilde}) implies that $\widetilde{\Lambda}$ is
a non decreasing function and that for $p+q\leq 2$,
\begin{equation}\label{e:AL}
|\alpha_{p,q}(s)-\alpha_{p,q}(r)|\leq|\widetilde{\Lambda}(s)-\widetilde{\Lambda}(r)|\;,
\; \textrm{for all } r, s\; .
\end{equation}
Since $p+q\leq 2$ and $m=1$, we need to consider only $p=1$, $q=0$ and
$p=0$, $q=1$ in (\ref{def:W_n}), we thus get
\begin{equation}\label{e:WL}
\widetilde{W}_n(s+\delta)-\widetilde{W}_n(t)\leq
2(\widetilde{\Lambda}(s+\delta)-\widetilde{\Lambda}(s))\;
|\sum_{1\leq i\neq j\leq n} X_i|\; .
\end{equation}
In the same way, after switching $s$ and $t$, we obtain
\begin{multline*}
\widetilde{R}_n(s)-\widetilde{R}_n(t)\leq \widetilde{R}_n(s+\delta)-\widetilde{R}_n(s)
+2(\widetilde{\Lambda}(s+\delta)-\widetilde{\Lambda}(s))
|\sum_{1\leq i\neq j\leq n} X_i|\\+2 n(n-1)(U(s+\delta)-U(s))\; ,
\end{multline*}
which gives \eqref{eq:decomp_Rn_tilde_1}.
In Lemma 5.2, P. 4307 of
\cite{borovkova:burton:dehling:2001}, the monotone Lipschitz-continuous function
$\Lambda$ is here $2U$, the process $\{Y_n(t)\}$ is
$\{a_n\widetilde{\Lambda}(t)(\sum_{1\leq i\neq j\leq n} X_i)\}$ if $m=1$,
and if $m=2$, the process $\{Y_n(t)\}$ is
$\{a_n\widetilde{\Lambda}(t)(\sum_{1\leq i\neq j\leq n} X_i X_j
+\sum_{1\leq i\neq j\leq n}(X_i^2-1))\}$.
Using Lemma \ref{lem:partie_qui_compte} and the fact that the function
$\widetilde{\Lambda}$ defined in (\ref{eq:lambda_tilde}) is a Lipschitz
function, the processes $\{Y_n(t),t\in I\}$ defined above satisfy
the condition (ii) of Lemma 5.2 in \cite{borovkova:burton:dehling:2001} with $r=2$.
Using Lemma \ref{lem:reste},
the condition (i) of Lemma 5.2 in \cite{borovkova:burton:dehling:2001}
is also satisfied. This concludes the proof.
\end{proof}

\section{Numerical experiments}\label{sec:exp}

In this section,  we investigate the robustness properties of the
Hodges-Lehmann and Shamos scale estimators defined in Section
\ref{sec:appli} using Monte Carlo experiments.
We shall regard the observations $X_t$, $t=1,\dots,n$, as a
stationary series $Y_t$, $t=1,\dots,n$, corrupted by additive
outliers of magnitude $\omega$. Thus we set 
\begin{align}\label{eq:ao}
X_t=Y_t+\omega W_t,
\end{align}
where $W_{t}$ are i.i.d. random variables. In Section
\ref{sub_sec:HL},  $W_{t}$ are Bernoulli($p/2$) random variables.
In Section \ref{sub_sec:SB}, $W_{t}$ are such that 
$\PP\left(W_t=-1\right)=\PP\left(W_t=1\right)$ $=p/2$ and
$\PP\left(W_t=0\right)=1-p$, hence $\mathbb{E}[W_t]=0$ and
$\mathbb{E}[W_t^2]=\Var(W_t)=p$. Observe that, in this case, $W$ is the product of Bernoulli($p$)
and \textit{Rademacher} independent random variables; the latter equals $1$ or $-1$, both with
probability $1/2$.
$(Y_t)_t$ is a stationary time series and it is assumed that $Y_t$ and $W_{t}$ are independent random variables.
The empirical study is based on 5000 independent replications  with
$n=600$, $p=10\%$ and $\omega= 10$.  
We consider the cases where $(Y_t)$ are Gaussian
ARFIMA$(1,d,0)$ processes, that is, 
\begin{equation}\label{e:FARIMA}
Y_t=(I-\phi B)^{-1}(I-B)^{-d} Z_t\; ,
\end{equation}
where $B$ denotes the backward operator, $\phi=0.2$ and $d=0.1$, 0.35,
corresponding respectively to 
$D=0.8$, 0.3, where $D$ is defined in (A\ref{assum:long-range}) and
$(Z_t)$ are i.i.d $\mathcal{N}(0,1)$.

\subsection{Hodges-Lehmann estimator}\label{sub_sec:HL}

In this section, we illustrate the results of
Proposition \ref{prop:hodges-lehmann}. In Figure
\ref{fig:hodg_lehm_no_out}, the empirical density functions
of $\hat{\theta}_{HL}$ and  $\bar{X}_n$
are displayed when $X_t$ has no outliers with $d=0.1$ (left) and
$d=0.35$ (right). 
In these cases both shapes are similar
to the limit indicated in Proposition \ref{prop:hodges-lehmann}, that
is, a Gaussian density with mean zero.

\begin{figure}[!h]
\begin{center}
\begin{tabular}{cc}
\includegraphics[width=0.52\textwidth]{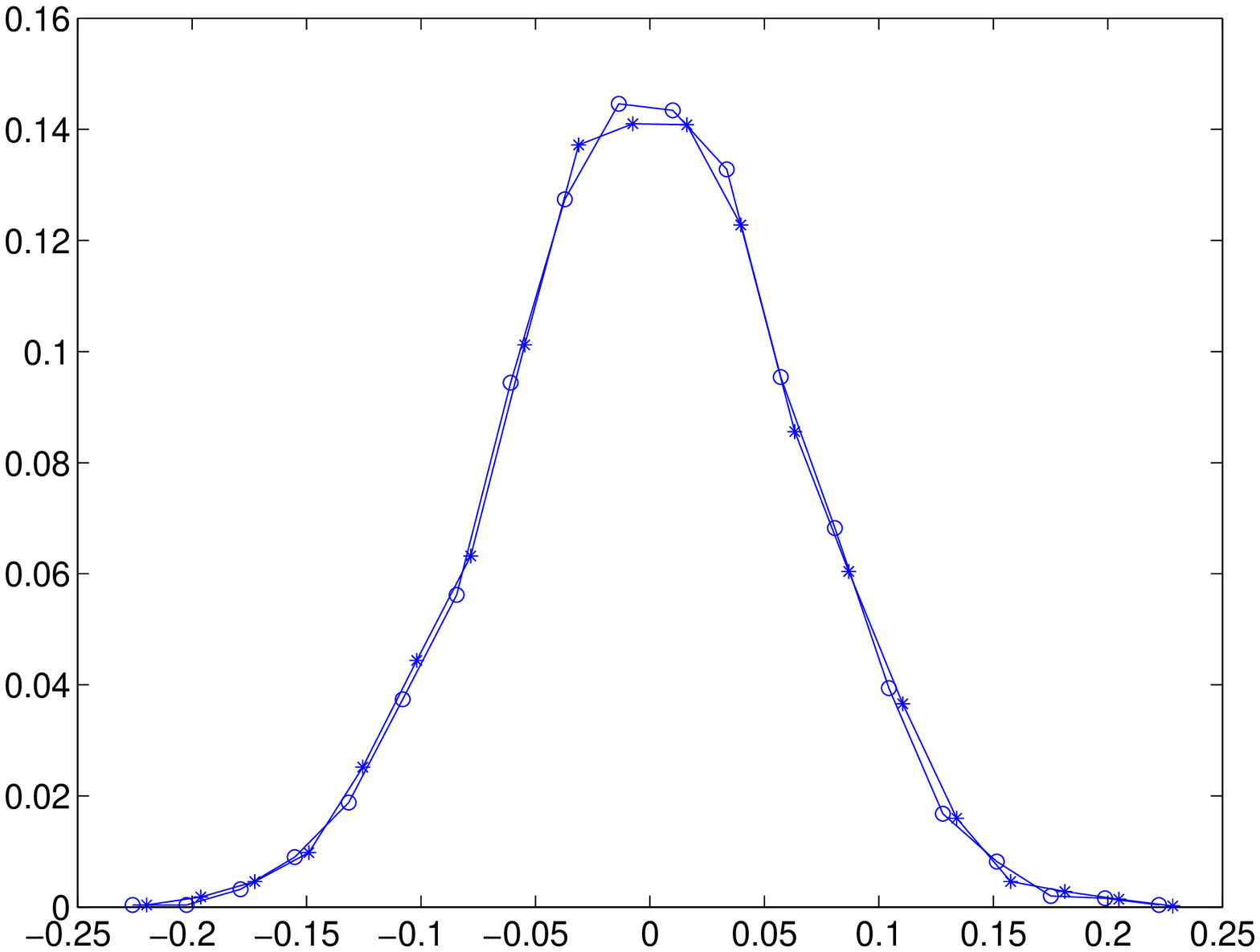}
&\hspace{-8mm}\includegraphics[width=0.48\textwidth]{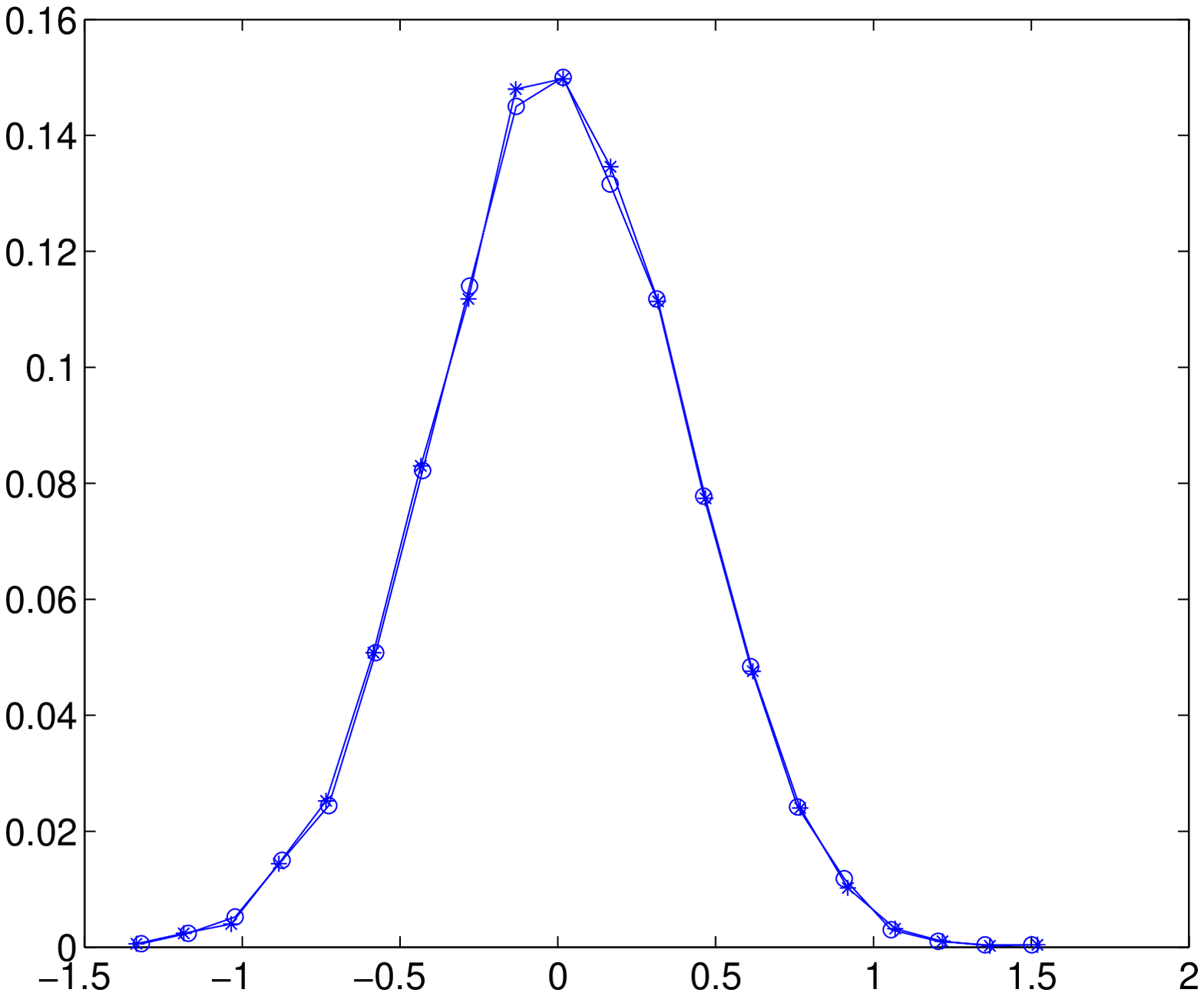}
\end{tabular}
\vspace{-7mm}
\caption{\footnotesize{Empirical densities of the
   quantities $\hat{\theta}_{HL}$
   ('*') and
   $\bar{X}_n$ ('o')  for the
   ARFIMA$(1,d,0)$ model with $d=0.1$ (left), $d=0.35$ (right), $n=600$ without
   outliers.}}
\label{fig:hodg_lehm_no_out}
\end{center}
\end{figure}

Figure \ref{fig:hodg_lehm_out} displays the same quantities as in
Figure \ref{fig:hodg_lehm_no_out} when $X_t$ has outliers with
$d=0.1$ (left) and $d=0.35$ (right). 
As expected, the sample mean is much more sensitive to
the presence of outliers than the Hodges-Lehmann estimator.
Observe that when the long-range dependence is strong (large $d$), the
effect of outliers is less pronounced.

\begin{figure}[!h]
\begin{center}
\begin{tabular}{cc}
\includegraphics[width=0.52\textwidth]{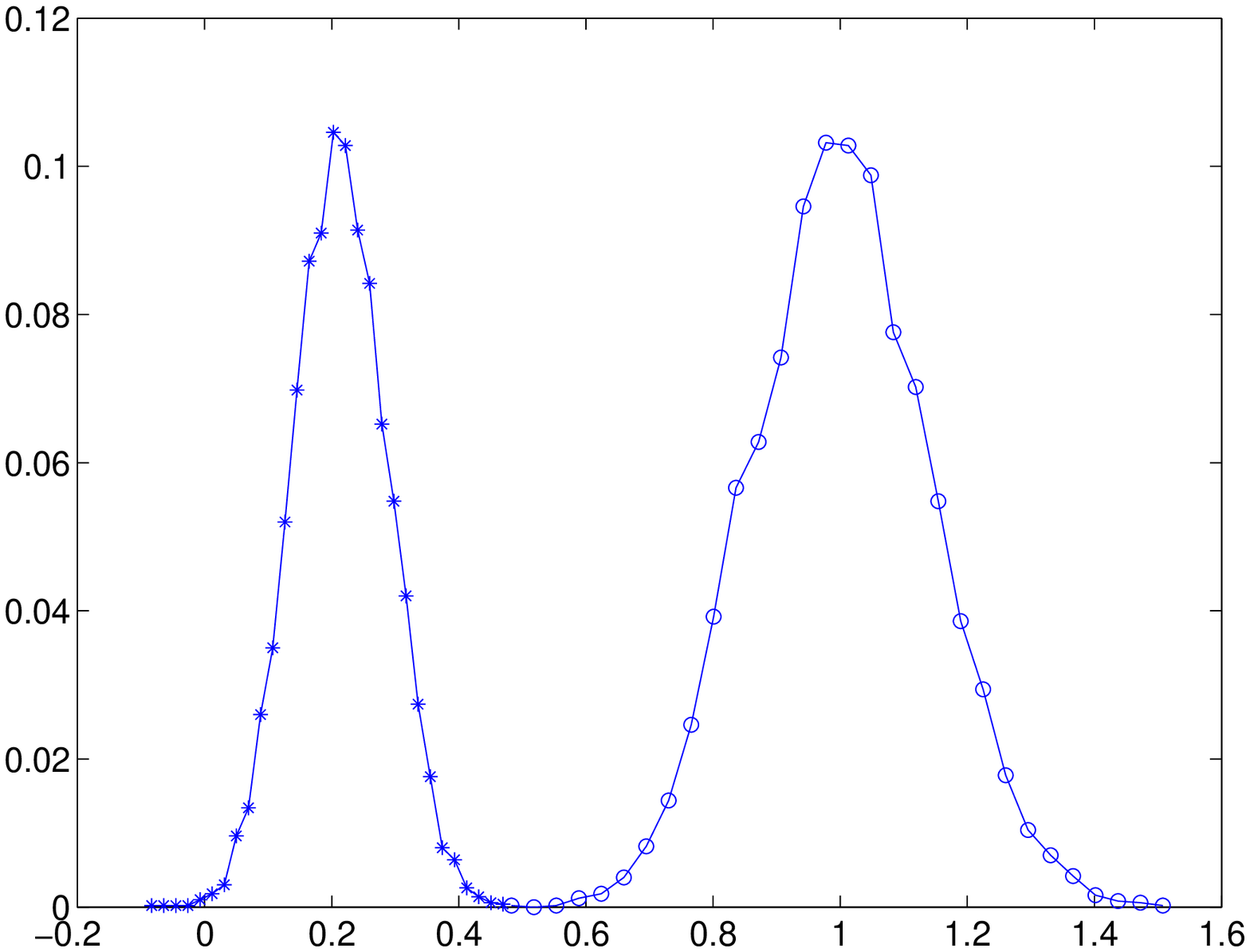}
&\hspace{-8mm}\includegraphics[width=0.52\textwidth]{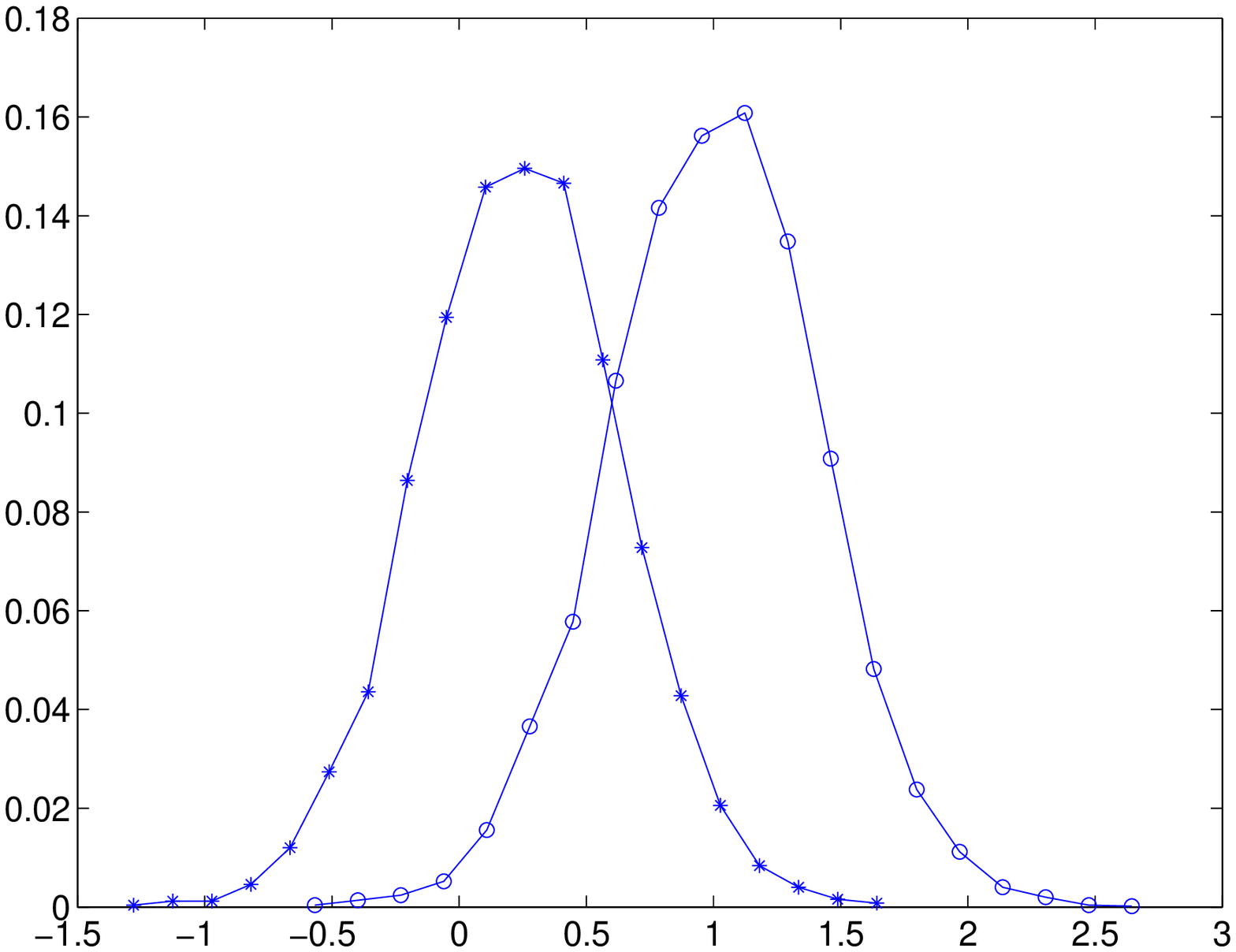}
\end{tabular}
\vspace{-7mm}
\caption{\footnotesize{Empirical densities of the
   quantities $\hat{\theta}_{HL}$
   ('*') and $\bar{X}_n$ ('o')  for the
   ARFIMA$(1,d,0)$ model with $d=0.1$ (left), $d=0.35$ (right), $n=600$ with
   outliers ($p=10\%$ and $\omega=10$).}}
\label{fig:hodg_lehm_out}
\end{center}
\end{figure}

\subsection{Shamos  scale estimator}\label{sub_sec:SB}

In this section, we illustrate the results of
Proposition \ref{prop:shamos_scale}.
In Figure \ref{fig:sham_bick_d02}, the empirical densities of
$\hat{\sigma}_{BL}-\sigma$ and $\hat{\sigma}_{n,X}-\sigma$
are displayed when $d=0.1$ without outliers (left) and with outliers
(right). In the left part of this figure, we illustrate the results
of the first part of Proposition \ref{prop:shamos_scale} since both shapes are similar to
that of Gaussian density with mean zero. On the right part
of Figure \ref{fig:sham_bick_d02}, we can see that the classical scale
estimator is much more sensitive to the presence of outliers than
the Shamos-Bickel estimator.

\begin{figure}[!h]
\begin{center}
\begin{tabular}{cc}
\includegraphics[width=0.52\textwidth]{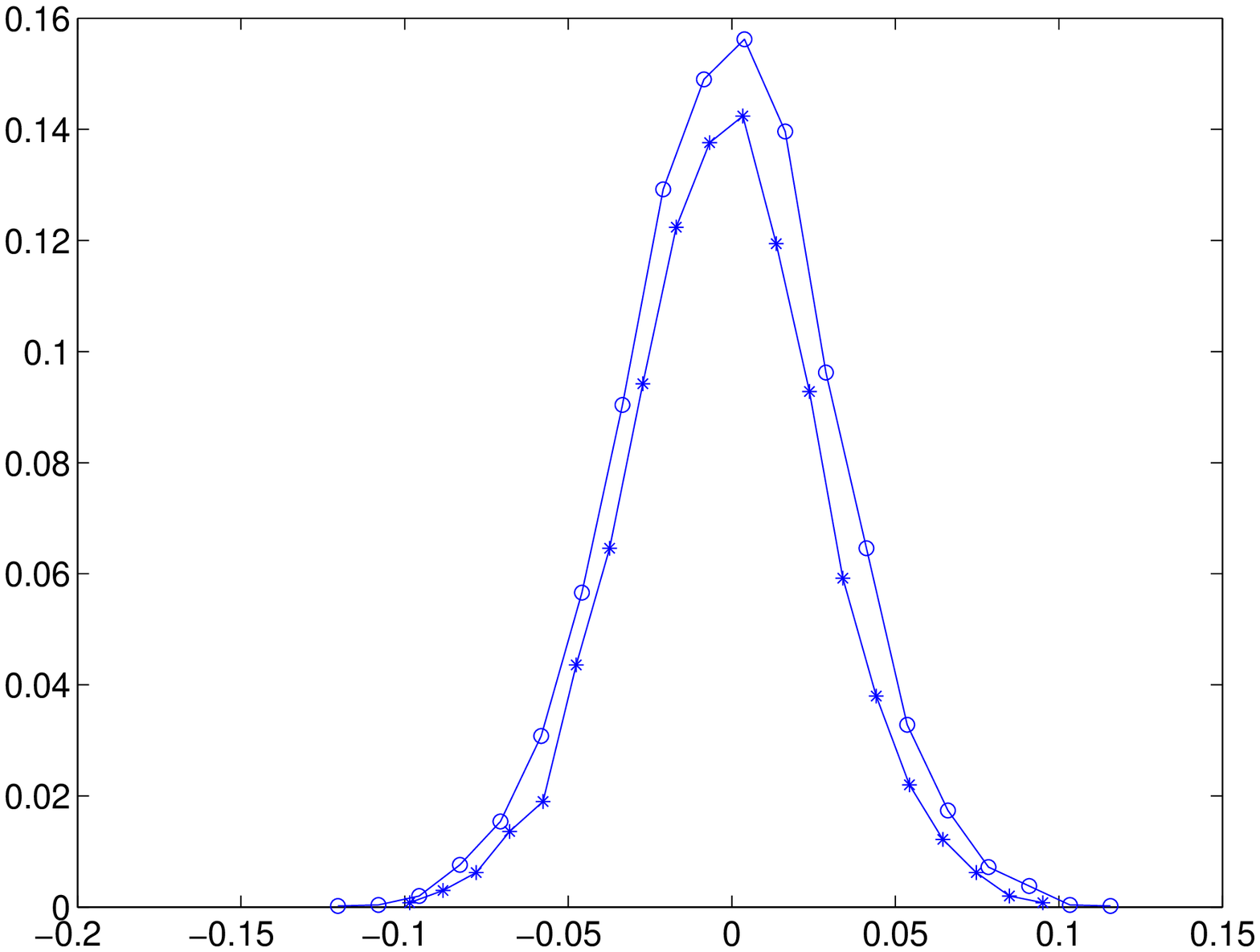}
&\hspace{-8mm}\includegraphics[width=0.52\textwidth]{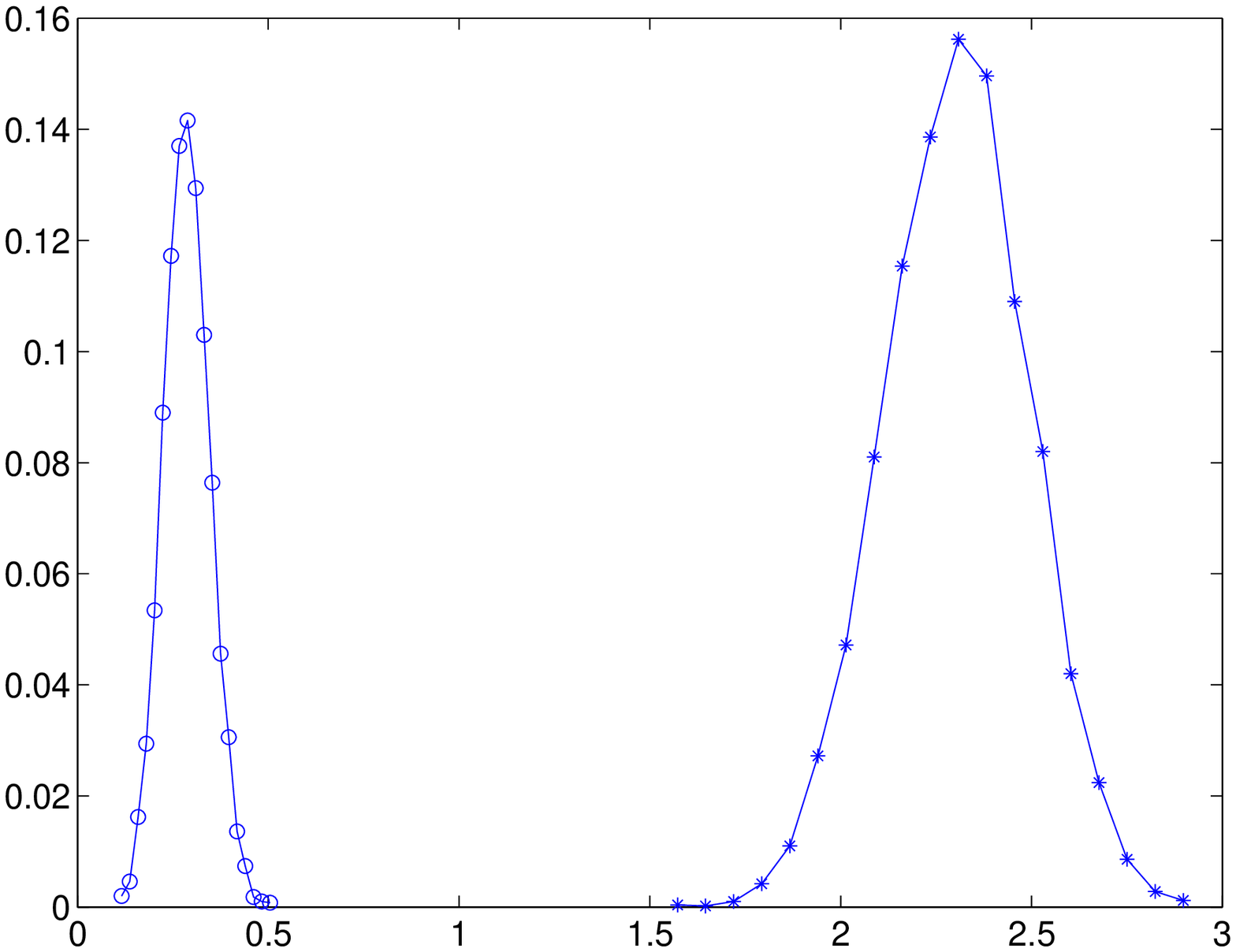}
\end{tabular}
\vspace{-7mm}
\caption{\footnotesize{Empirical densities of the
   quantities $(\hat{\sigma}_{BL}-\sigma)$
   ('*') and
   $(\hat{\sigma}_{n,X}-\sigma)$ ('o')  for the
   ARFIMA$(1,d,0)$ model with $d =0.1$, $n=600$ without
   outliers (left) and with outliers $p=10\%$ and $\omega=10$
   (right).}}
\label{fig:sham_bick_d02}
\end{center}
\end{figure}

\begin{figure}[!h]
\begin{center}
\begin{tabular}{cc}
\includegraphics[width=0.52\textwidth]{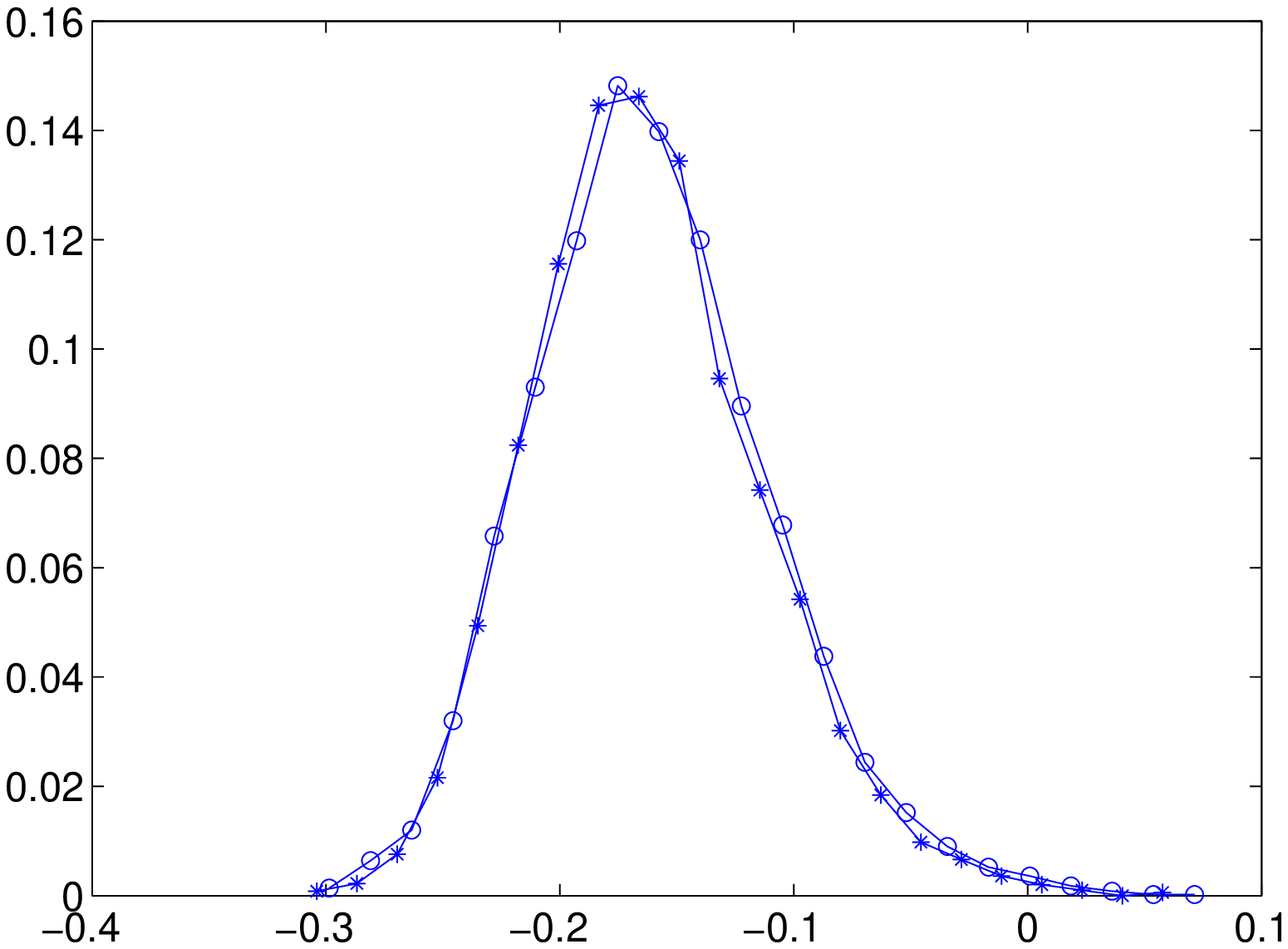}
&\hspace{-8mm}\includegraphics[width=0.52\textwidth]{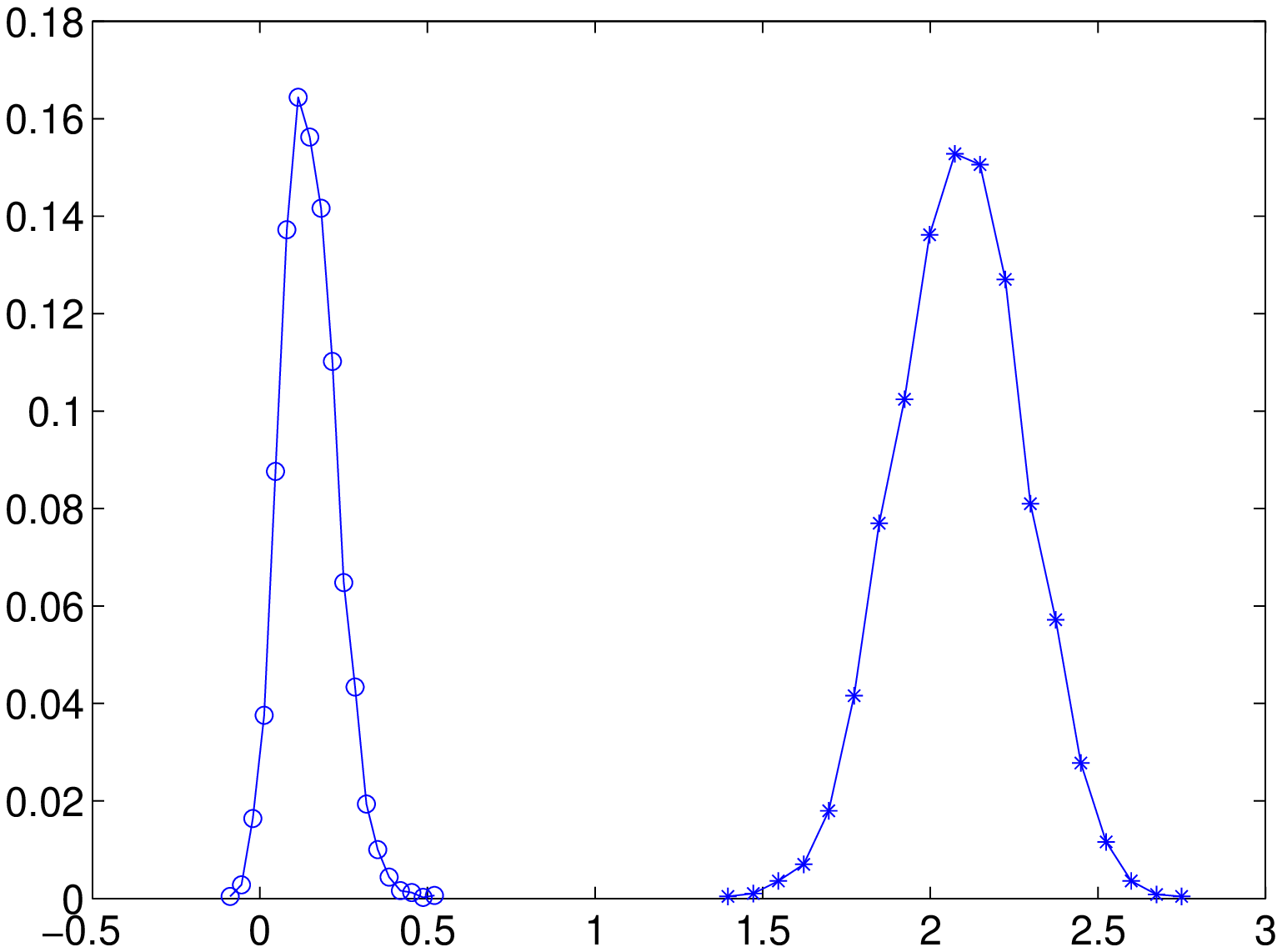}
\end{tabular}
\vspace{-7mm}
\caption{\footnotesize{Empirical densities of the
   quantities $(\hat{\sigma}_{SB}-\sigma)$
   ('*') and
   $(\hat{\sigma}_{n,X}-\sigma)$ ('o')  for the
   ARFIMA$(1,d,0)$ model with $d=0.35$, $n=600$ without
   outliers (left) and with outliers $p=10\%$ and $\omega=10$
   (right).}}
\label{fig:sham_bick_d045}
\end{center}
\end{figure}

Figure \ref{fig:sham_bick_d045} (left) illustrates the second part of 
Proposition \ref{prop:shamos_scale}. $d=0.35$ corresponds to
$D=0.3<1/2$. The right part of Figure
\ref{fig:sham_bick_d045} shows the robustness of the Shamos
estimator with respect to the classical scale estimator in the
presence of outliers.

\end{document}

%% file: def_aos.tex
\usepackage{amsfonts,amsmath,amssymb,amsthm,enumerate,color,hyperref,bbm,tabularx,ifthen,twoopt}
\usepackage{graphicx}
\usepackage[latin1]{inputenc}
\usepackage{natbib}

\def\rset{\mathbb{r}}

\newcounter{hypA}
\newenvironment{hypA}{\refstepcounter{hypA}\begin{itemize}
  \item[({\bf A\arabic{hypA}})]}{\end{itemize}}

\def\rmd{\mathrm{d}}

\def\PE{\mathbb{E}}

\def\PP{\mathbb{P}}

\def\Cov{\mathrm{Cov}}

\def\Var{\mathop{\rm Var}\nolimits}

\def\rmd{\mathrm{d}}
\def\trace{\textrm{Tr}}

\newcommandtwoopt{\QRC}[4][][]
{\ifthenelse{\equal{#1}{}}
{\ifthenelse{\equal{#2}{}}{\mathrm{Q}_{#3}\left( #4\right)}{\mathrm{Q}_{#3}\left(#4,#2\right)}}
{\ifthenelse{\equal{#2}{}}{\mathrm{Q}^{#1}_{#3}\left(#4\right)}{\mathrm{Q}^{#1}_{#3}\left(#4,#2\right)}
}
}

\newtheorem{theo}{Theorem}
\newtheorem{prop}[theo]{Proposition}
\newtheorem{lemma}[theo]{Lemma}
\newtheorem{corollary}[theo]{Corollary}
\theoremstyle{remark}
\newtheorem{remark}{Remark}

\def\rset{\mathbb{R}}

\def \1{\mathbbm{1}}
\def\B{\textrm{B}}